%% file: AOV22-FINAL.tex
\documentclass[12pt,a4paper,reqno]{amsart}

\usepackage[margin=2.8cm,footskip=1cm]{geometry}
\usepackage[utf8]{inputenc}
\usepackage[T1]{fontenc}
\usepackage[english]{babel}
\usepackage{amsmath}
\usepackage{amsfonts}
\usepackage{amssymb}
\usepackage{amsthm}
\usepackage{graphicx}
\usepackage{xcolor}
\usepackage{dsfont}
\usepackage{physics}
\usepackage{mathtools}
\usepackage{extarrows}
\usepackage{tikz,pgfplots}
\usepackage{mathrsfs} 
\usepackage[hidelinks]{hyperref}
\usepackage[font={small}]{caption}
\graphicspath{{Images/}}

\usepackage{constants}

\newconstantfamily{ctrlcst}{
	symbol=\kappa
}
\newconstantfamily{csts}{
	symbol=C
}
\renewconstantfamily{normal}{
	symbol=c
}

\makeatletter
\newcommand{\pushright}[1]{\ifmeasuring@#1\else\omit\hfill\(\displaystyle#1\)\fi\ignorespaces}
\newcommand{\pushleft}[1]{\ifmeasuring@#1\else\omit\(\displaystyle#1\)\hfill\fi\ignorespaces}
\makeatother

\newcommand{\betasat}{\beta_{\mathrm{sat}}}
\newcommand{\surcharge}{\mathfrak{s}}
\newcommand{\Wulff}{\mathscr{W}}
\newcommand{\betahat}{\hat{\beta}_{\mathrm{sat}}}

\renewcommand{\norm}[1]{\|#1\|}
\newcommand{\normI}[1]{\left\|#1\right\|_{\scriptscriptstyle 1}}
\newcommand{\normsup}[1]{\left\|#1\right\|_{\scriptscriptstyle\infty}}

\newcommand{\setof}[2]{\{#1\,:\,#2\}}
\newcommand{\bsetof}[2]{\bigl\{#1\,:\,#2\bigr\}}
\newcommand{\Bsetof}[2]{\Bigl\{#1\,:\,#2\Bigr\}}
\newcommand{\comp}{\mathrm{c}}

\newcommand{\IF}[1]{\mathds{1}_{\{#1\}}}

\newcommand{\bbG}{\mathbb{G}}

\newcommand{\bbJ}{\mathbb{J}}

\newcommand{\bbN}{\mathbb{N}}

\newcommand{\bbP}{\mathbb{P}}

\newcommand{\bbR}{\mathbb{R}}
\newcommand{\bbS}{\mathbb{S}}

\newcommand{\bbZ}{\mathbb{Z}}

\newcommand{\calC}{\mathcal{C}}

\newcommand{\calF}{\mathcal{F}}

\newcommand{\calW}{\mathcal{W}}

\newcommand{\rmd}{\mathrm{d}}

\newcommand{\sfe}{\mathsf{e}}
\newcommand{\sfo}{\mathsf{o}}
\newcommand{\sfO}{\mathsf{O}}

\newcommand{\lambdasat}{\lambda_{\mathrm{sat}}}
\newcommand{\alphasat}{\alpha_{\mathrm{sat}}}

\newcommand{\cone}{\mathcal{Y}}

\newcommand{\Tc}{T_{\mathrm{\scriptscriptstyle c}}}

\newcommand{\Piv}{\mathrm{Piv}}

\newcommand{\betaexp}{\beta_{\mathrm{exp}}}
\newcommand{\betac}{\beta_{\mathrm{\scriptscriptstyle c}}}

\newcommand{\icl}{\nu}

\newcommand{\niceConnection}{NC}

\newcommand{\constPrefact}{\tilde{\chi}}

\newcommand{\Z}{\mathbb{Z}}
\newcommand{\R}{\mathbb{R}}
\newcommand{\Rd}{\mathbb{R}^d}
\newcommand{\Zd}{\mathbb{Z}^d}

\newcommand{\given}{\,|\,}
\newcommand{\bgiven}{\bigm\vert}

\renewcommand{\nleftrightarrow}{\mathrel{\ooalign{\(\leftrightarrow\)\cr\hidewidth\(/\)\hidewidth}}}

\theoremstyle{plain}
\newtheorem{theorem}{Theorem}[section]
\newtheorem{lemma}[theorem]{Lemma}

\newtheorem{corollary}[theorem]{Corollary}
\newtheorem{conjecture}[theorem]{Conjecture}
\newtheorem*{conjecture*}{Conjecture}
\newtheorem{remark}{Remark}[section]

\newtheorem{claim}{Claim}

\theoremstyle{definition}
\newtheorem{definition}{Definition}[section]
\newtheorem{obs}{Observation}

\author{Yacine Aoun}
\address{Section de Mathématiques, Université de Genève, CH-1211 Genève, Switzerland}
\email{Yacine.Aoun@unige.ch}
\author{S\'{e}bastien Ott}
\address{Département de Mathématiques, Université de Fribourg, Chemin du Musée 23, 1700 Fribourg, Switzerland}
\email{ott.sebast@gmail.com}
\author{Yvan Velenik}
\address{Section de Mathématiques, Université de Genève, CH-1211 Genève, Switzerland}
\email{Yvan.Velenik@unige.ch}

\date{\today}

\title[On the Potts 2-point function in the saturation regime]{On the two-point function of the Potts model\\in the saturation regime}

\begin{document}

\begin{abstract}
	We consider the Random-Cluster model on \(\Zd\) with interactions of infinite range of the form \(J_x = \psi(x)\sfe^{-\rho(x)}\) with \(\rho\) a norm on \(\Zd\) and \(\psi\) a subexponential correction. We first provide an optimal criterion ensuring the existence of a nontrivial \emph{saturation} regime (that is, the existence of \(\betasat(s)>0\) such that the inverse correlation length in the direction \(s\) is constant on \([0,\betasat(s)\))), thus removing a regularity assumption used in our previous work~\cite{Aoun+Ioffe+Ott+Velenik-CMP-2021}. Then, under suitable assumptions, we derive sharp asymptotics (which are not of Ornstein--Zernike form) for the two-point function in the whole saturation regime \((0,\betasat(s))\). We also obtain a number of additional results for this class of models, including sharpness of the phase transition, mixing above the critical temperature and the strict monotonicity of the inverse correlation length in \(\beta\) in the regime \((\betasat(s), \betac)\).
\end{abstract}

\maketitle


\section{Introduction}

Since the celebrated work of Ornstein and Zernike more than a century ago~\cite{Ornstein+Zernike-1914, Zernike-1916}, the analysis of the asymptotic behavior of correlation functions has played an important role in our understanding of the equilibrium properties
of macroscopic systems. In particular, Ornstein and Zernike predicted that the pair correlation function \(G_\beta(r)\) would decay, as a function of the distance \(r\), according to \(r^{-(d-1)/2}\sfe^{-\icl_\beta r}\), where \(\icl_\beta\) denotes the inverse correlation length. This has since become known as \emph{Ornstein--Zernike (OZ) behavior} and is expected to be the generic behavior away from critical points.

While their original work relied on unproven (both explicit and implicit) assumptions, the validity of OZ behavior was established rigorously in a variety of settings (see also~\cite{Ott+Velenik-2019} for a more detailed overview restricted to the Ising model): the first derivation of OZ behavior was done in the planar Ising model above the critical temperature~\cite{Wu-1966,Wu+McCoy+Tracy+Barouch-1976} by explicitly computing the pair correlation function (these computations also showed that OZ behavior is violated below the critical temperature in the planar Ising model); then, OZ behavior was established in more general systems and in any dimension in perturbative regimes starting with the works~\cite{Abraham+Kunz-1977, Paes-Leme-1978, Bricmont+Frohlich-1985a}; the first non-perturbative approaches were introduced in the 1980s, but were restricted to simple models (in particular, the self-avoiding walk~\cite{Chayes+Chayes-1986, Ioffe-1998} and Bernoulli percolation~\cite{Campanino+Chayes+Chayes-1991}), were very model-dependent and lacked robustness. In the last two decades, a much more powerful and robust approach was developed in the works~\cite{Campanino+Ioffe-2002, Campanino+Ioffe+Velenik-2003, Campanino+Ioffe+Velenik-2008, Ott+Velenik-2018}. Among others, the latter version of the theory allowed the analysis of the odd-odd and even-even correlations in the finite-range Ising model on \(\Zd\) above \(\Tc\)~\cite{Campanino+Ioffe+Velenik-2004, Ott+Velenik-2018(2)}, of the 2-point correlation function of the finite-range Ising model on \(\Zd\) in a field at any temperature~\cite{Ott-2020} and of the 2-point correlation function of the finite-range Potts model on \(\Zd\) above \(\Tc\)~\cite{Campanino+Ioffe+Velenik-2008}, including in the presence of inhomogeneities~\cite{Ott+Velenik-2018}.

At this stage, it was natural to proceed one step further by considering models with interactions of infinite range. It had been understood since at least the 1960s~\cite{Widom-1964} that the pair correlation function cannot decay faster than the pair interaction (at least in ferromagnetic-type systems). In particular, OZ behavior cannot occur when the interactions decay slower than exponentially; the sharp asymptotic behavior of the 2-point function in this regime was established for the Ising model on \(\Zd\) above \(\Tc\) in~\cite{Newman+Spohn-1998} and (in the case of interactions decaying according to a power law) for the Potts model on \(\Zd\) above \(\Tc\) in~\cite{Aoun-2021} (the extension of the latter result to interactions decaying like a stretched exponential is provided in the present work, see below). However, it was expected (see, for instance, \cite{Bricmont+Frohlich-1985b}) that OZ behavior should occur (at least at sufficiently high temperatures) whenever the interaction decays at least exponentially fast with the distance. This turns out to be incorrect, as we explain now.

\bigskip
In~\cite{Aoun+Ioffe+Ott+Velenik-CMP-2021} (see also~\cite{Aoun+Ioffe+Ott+Velenik-PRE-2021}), we considered a general class of lattice spin systems on \(\Zd\) with two-body ferromagnetic interactions decaying asymptotically as \(\psi(x)\sfe^{-\rho(x)}\), where \(\rho\) is a norm on \(\Rd\) and \(\psi\) is a sub-exponential correction. Let us denote by \(G_\beta(x)\) the associated two-point function and by \(\nu_\beta(s)\) the corresponding inverse correlation length in direction \(s\), defined as the rate of exponential decay of \(G_\beta\): \(G_\beta(ns) = \sfe^{-\nu_\beta(s)n + \sfo(n)}\) as \(n\to\infty\) and \(s\in\bbS^{d-1}\). It is easy to check that \(\beta\mapsto\nu_\beta(s)\) is non-increasing and that \(\lim_{\beta\downarrow 0} \nu_\beta(s) = \rho(s)\) for all \(s\in\bbS^{d-1}\). This leads naturally to the introduction of the \emph{saturation point} \(\betasat(s) = \sup\setof{\beta\geq 0}{\nu_\beta(s) = \rho(s)}\) (see Figure~\ref{fig:saturation}). We call \((0, \betasat(s))\) the \emph{saturation} regime (in direction \(s\)). Our goal in~\cite{Aoun+Ioffe+Ott+Velenik-CMP-2021} was to understand under which conditions saturation does occur, that is, when is \(\betasat(s)\) strictly positive.
 This turns out to depend on the direction \(s\), the norm \(\rho\) and the prefactor \(\psi\). The main result in~\cite{Aoun+Ioffe+Ott+Velenik-CMP-2021} was the derivation of a criterion characterizing exactly the prefactors \(\psi\) leading to \(\betasat(s)>0\) (actually, in~\cite{Aoun+Ioffe+Ott+Velenik-CMP-2021}, we imposed a mild regularity condition on the behavior of the norm \(\rho\) in the neighborhood of the direction \(s\); removing this condition is one of the results of the present paper, as we explain below). Note that, when \(\betasat(s)>0\), the correlation length is not an analytic function of the temperature in the high-temperature regime; that this could occur (even in dimension \(1\)!) was also unexpected.

\begin{figure}[t]
	\centering
	\includegraphics{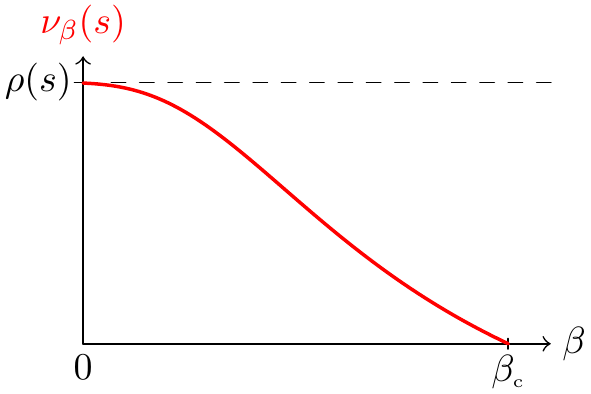}
	\hspace{2cm}
	\includegraphics{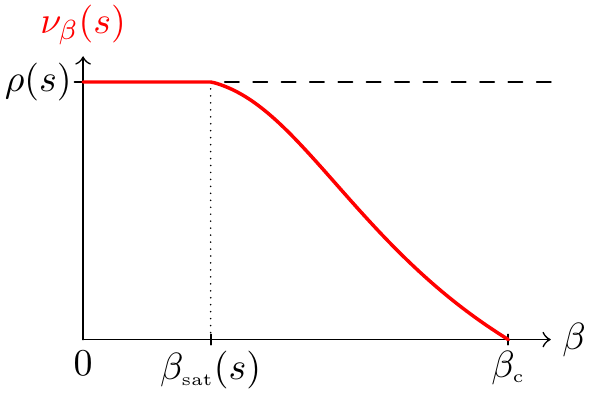}
	\caption{The two possible behaviors of the inverse correlation length in a direction \(s\in\bbS^{d-1}\) in the Ising model on \(\Zd\) with coupling constants \(J_x = \psi(x)\sfe^{-\rho(x)}\). Left: saturation never occurs (\(\betasat(s)=0\)). Right: saturation occurs (\(\betasat(s)>0\)). In more general models, the behavior is completely similar, except that the inverse correlation length does not necessarily converge to \(0\) at \(\betac\).}
	\label{fig:saturation}
\end{figure}

In~\cite{Aoun+Ioffe+Ott+Velenik-CMP-2021}, we also described (with partial results) how the asymptotic behavior of the two-point function is expected to change depending on whether saturation occurs. In the regime \((\betasat(s),\betac)\), \(G_\beta\) is expected to always display standard Ornstein--Zernike asymptotics: \(G_\beta(x) \sim \abs{x}^{-(d-1)/2} \sfe^{-\nu_\beta(x)}\). This is however expected not to be true in the regime  \((0,\betasat(s))\) in which saturation occurs; here, \(G_\beta(x)\) (with \(x\),\(s\) colinear) is dominated by the direct interaction between the spins in the neighborhood of \(0\) and those in a neighborhood of \(x\), which should generally lead to \(G_\beta(x) \sim \psi(x)\sfe^{-\rho(x)}\). This was proved for sufficiently small values of \(\beta\) in a variety of models in~\cite{Aoun+Ioffe+Ott+Velenik-CMP-2021}. This difference in behavior should be related to a drastic change in the morphology of typical paths contributing to the high-temperature expansion of \(G_\beta(x)\), analogous to what happens in condensation phenomena for sums of independent random variables~\cite{Godreche-2019}; see Fig.~\ref{fig:CondensationPaths}.

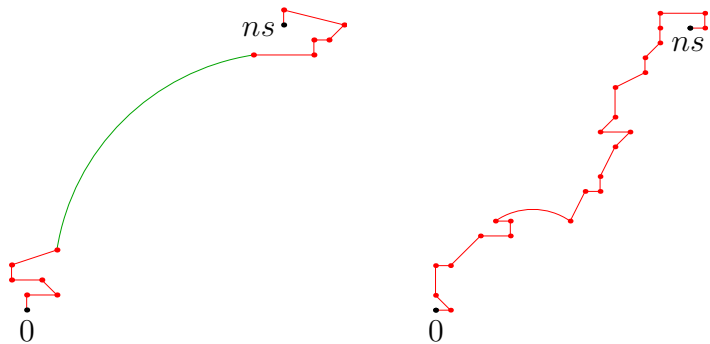
\begin{figure}[ht]
	\centering
	\begin{tikzpicture}
		\node[anchor=south west,inner sep=0] at (0,0) {\includegraphics[width=4.5cm]{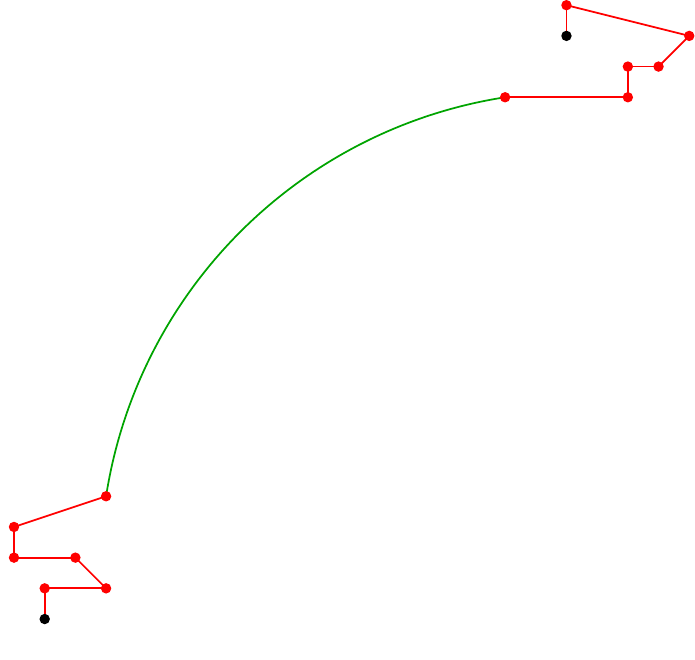}\hspace{1cm}
			\raisebox{2.3mm}{\includegraphics[height=4cm]{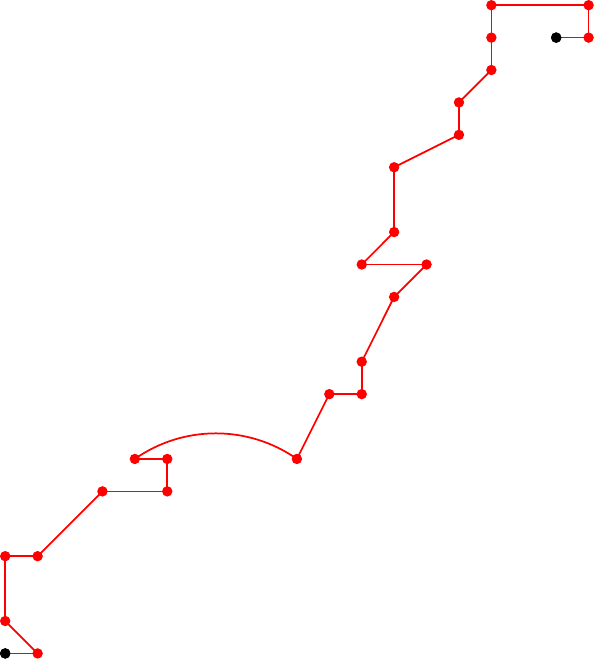}}};
		\node[left] at (0.55,0) {$0$};
		\node at (3.35,4.0) {$n s$};
		\node[left] at (5.93,0) {$0$};
		\node at (9,3.8) {$n s$};
	\end{tikzpicture}
	\caption{Qualitative depiction of typical paths contributing to the high-temperature expansion of the 2-point function \(G_\beta(x)\) of the Ising model (with \(\norm{x}\gg 1\)). Left: In the saturation regime, there is a single giant edge connecting a vertex close to \(0\) to a vertex close to \(x\). Right: In the regime \((\betasat(s),\betac)\), all edges are microscopic (and the path can in fact be coupled, using the Ornstein--Zernike theory) to a directed random walk.}
	\label{fig:CondensationPaths}
\end{figure}

\bigskip
The discussion above applies to a general class of lattice spin systems. In order to obtain more precise results, it is useful to turn to important specific examples. In~\cite{Aoun+Ott+Velenik-2021}, we considered the ferromagnetic Ising model on \(\Zd\) with coupling constants of the form \(J_x = \psi(x)\sfe^{-\rho(x)}\) as above. Extending earlier results restricted to finite-range interactions~\cite{Campanino+Ioffe+Velenik-2003} (see also \cite{Ott+Velenik-2019} for an overview), we proved that the two-point function indeed exhibits Ornstein--Zernike behavior for all \(\beta\in (\betasat(s),\betac)\) (under some regularity condition). This provides a precise (although not yet completely exhaustive) description of the regime in which saturation does not occur.

\medskip
In the present work, we consider the saturation regime.
Our main goal is to provide a detailed analysis of the two-point function  in this regime. We do that by analyzing the \(q\geq 1\) Random-Cluster model on \(\Zd\) (which includes Bernoulli percolation and the Ising and Potts models on \(\Zd\)).
We introduce the required terminology and notation in Section~\ref{sec:Notations} and then state our main results in Section~\ref{sec:Results}.

\section{Model and notations}\label{sec:Notations}

Most of our results naturally extend to a wider set-up but we restrict attention to \(\Zd\). We will always see \(\Zd\) as canonically embedded inside \(\bbR^d\) and will denote \(\norm{\cdot}\) the Euclidean norm on \(\Rd\). \(\rho\) will denote a norm on \(\bbR^d\) (and will be one of the parameters in our analysis).

We consider the graph \((\Zd,E_d)\) with edge set \(E_d = \bigl\{\{i,j\}\subset\Zd\bigr\}\), which we will often write simply \(\Zd\). Let \(\Lambda_N=\{-N,\dots,N\}^d\) and \(\Lambda_N(x) = x+\Lambda_N\). For \(x,y\in\Zd\), we denote by \([x,y]\) the closed line segment in \(\Rd\) with endpoints \(x\) and \(y\).

\subsection{Graphs}

For a graph \((V,E)\), \(A\subset V\) and \(F\subset E\), we write \(A^\comp=V\setminus A\),
\begin{gather*}
	E_A = \bsetof{\{i,j\}\in E}{\{i,j\}\subset A},\quad
	\partial A = \bsetof{ \{i,j\}\in E}{i\in A, j\in A^\comp},\\
	V_F = \bigcup_{\{i,j\}\in F} \{i,j\},\quad
	\partial F = \bsetof{ \{i,j\}\in E}{i\in V_F}\setminus F,\quad
	\bar{E}_A = E_A \cup \partial E_A.
\end{gather*}

We will systematically identify sets and their characteristic function (e.g., \(\omega\subset E\) will be identified with \(\omega\in\{0,1\}^E\), where \(\omega_e=1\) if and only if \(e\in\omega\)).

In all this work, we will consider subgraphs of \((\Z^d,E_d)\). For \(\omega,\eta\subset E_d\), we will denote by \(\omega|_F\) the restriction of \(\omega\) to \(F\) and by \(\omega|_F\eta|_{F^\comp}\) the union of the edges in \(\omega|_F\) and in \(\eta|_{F^\comp}\). We endow the subsets of \(E_d\) with the usual partial order (that is, sets are ordered by inclusion).

\subsection{Random-Cluster model}

\subsubsection{Interaction}\label{ssec:Interactions}

We consider a weight function (the \textit{interaction}, or the set of \textit{coupling constants}) \(J:E_d\to\R_+\) satisfying
\begin{itemize}
	\item \textsf{Translation/reflection invariance:} \(J_{ij} = J_{i-j}= J_{j-i}\),
	\item \textsf{No self-interaction:} \(J_0=0\),
	\item \textsf{Normalization:} \(\sum_{x\in\Zd} J_x = 1\).
\end{itemize}
We will use the following terminology.
\begin{definition}
	\label{def:short_range}
	\(J\) is \emph{exponentially-bounded} if \(J_x\leq \sfe^{-c\norm{x}}\) for some \(c>0\) and all \(x\in\Zd\) with \(\norm{x}\) sufficiently large.
	\(J\) is \emph{exponentially-decaying} if \(J_x= \psi(x) \sfe^{-\rho(x)}\) with \(\psi>0\) satisfying
	\[
		\lim_{\norm{x}\to\infty} \dfrac{\log\psi(x)}{\norm{x}} = 0,
	\]
	and \(\rho\) a norm on \(\Rd\).
\end{definition}

\subsubsection{The Random-Cluster model}

We refer to~\cite{Grimmett-2006, Duminil-Copin-2017} for additional details on the Random-Cluster model.
Let \(q\geq 1\) and \(\beta\geq 0\). Consider a finite set \(F\subset E_d\). Let \(\eta \subset E_d\) be such that the graph \((\Zd,\eta)\) contains at most one infinite cluster.
The Random-Cluster measure on \(F\) with boundary condition \(\eta\) is the probability measure on \(\{0,1\}^F\) given by
\begin{equation}
	\Phi_{F;q,\beta}^{\eta}(\omega) = \frac{1}{Z_{F;q,\beta}^{\eta}} \prod_{e\in F}(\sfe^{\beta J_e}-1)^{\omega_e} q^{\kappa_{\eta}(\omega)},
\end{equation}
where \(\kappa_{\eta}(\omega)\) is the number of connected components in \((\Zd,\omega|_F\eta|_{F^\comp})\) having an endpoint in \(V_F\) and \(Z_{F;q,\beta}^{\eta}\) is the partition function.
For \(V\subset\Zd\), we set \(\Phi_{V;q,\beta}^{\eta}\equiv \Phi_{\bar{E}_V;q,\beta}^{\eta}\).
The following stochastic domination (with respect to the partial order on subsets of \(E_d\)) applies:
\[
	\forall\eta\leq\eta',\forall\beta\leq\beta',\qquad
	\Phi_{F;q,\beta}^{\eta} \preccurlyeq \Phi_{F;q,\beta'}^{\eta'}.
\]
In particular, the boundary conditions \(\eta\equiv 0\) and \(\eta \equiv 1\) are extremal; they will respectively be denoted by \(0\) and \(1\).
Moreover, the following finite energy property applies: for any $e\in F$ and $\eta'\in\lbrace 0,1\rbrace^{F\setminus\lbrace e\rbrace}$, one has
\begin{equation*}
\dfrac{\sfe^{\beta J_{e}}-1}{\sfe^{\beta J_{e}} -1+q}
\leq \Phi_{F;q,\beta}^{\eta}(\omega_{e}=1 \mid \omega_{F\setminus\lbrace e\rbrace}=\eta')\leq 1-\sfe^{-\beta J_{e}}
\end{equation*}
The limits
\[
	\Phi_{q,\beta}^{*} =\lim_{F\to E_d}\Phi_{F;q,\beta}^{*},
\]
exist for \(*\in\{0,1\}\) and are translation invariant.
Since \(q\geq 1\) will be kept fixed in our analysis, it will be removed from the notation.
For \(F\subset E_d\), denote by \(\calF_{F}\) the sigma-algebra generated by the edge variables \(\omega_e\), \(e\in F\). We say that an event \(A\) is \emph{supported on \(F\)} if \(A\in\calF_F\). As usual, we denote \(\{x\leftrightarrow y\}\) for the event ``\(x\) connected to \(y\)'', and \(\{x\xleftrightarrow[]{F} y\}\) for the event ``\(x\) connected to \(y\) using only edges in \(F\)''.

The \emph{two-point function} of the Random-Cluster model is
\[
	G_{\beta}(x,y) = \Phi_\beta^0(x\leftrightarrow y).
\]

When \(q\in\bbN\) with \(q\geq 2\), one has the following correspondance with the well-known \emph{Potts model}
\begin{equation}
	\bbP^{\mathrm{Potts}}_{\beta,q}(\IF{\sigma_x=\sigma_y})- \frac{1}{q} = \frac{q-1}{q} \, \Phi^{0}_{\beta,q}(x\leftrightarrow y).
\end{equation}
We refer to~\cite{Duminil-Copin-2017} for more details.

\subsection{Convex geometry}

It will be convenient to introduce a few quantities associated to the norm $\rho$. First, two convex sets are important: the unit ball \(\mathscr{U}\subset \bbR^d\) and the corresponding \emph{Wulff shape}
\[
	\Wulff = \setof{t\in\bbR^d}{\forall x\in\bbR^d,\, t\cdot x \leq \rho(x)}.
\]
Given a direction \(s\in \bbS^{d-1}\), we say that the vector \(t\in\bbR^d\) is dual (or \(\rho\)-dual) to \(s\) if
\(t\in\partial\Wulff\) and \(t\cdot s = \rho(s)\). A direction \(s\) possesses a unique dual vector \(t\) if and only if \(\partial\Wulff\) does not possess an affine part with normal \(s\). Equivalently, there is a unique dual vector when the unit ball $\mathscr{U}$ has a unique supporting hyperplane at \(s/\rho(s)\). (See Fig.~\ref{fig:duality} for an illustration.)

\begin{figure}[ht]
	\includegraphics{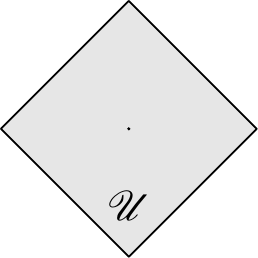}
	\hspace*{1cm}
	\includegraphics{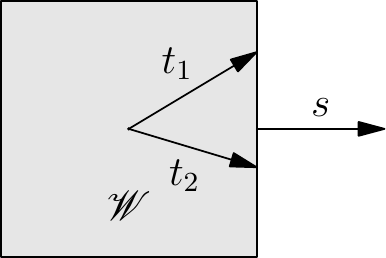}
	\hspace*{1cm}
	\includegraphics{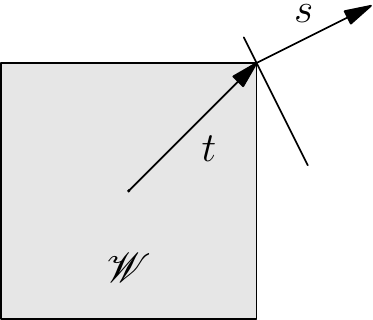}
	\caption{Left: The unit ball for the norm \(\rho(\cdot)=\normI{\cdot}\). Middle: the corresponding Wulff shape $\Wulff$ with two vectors \(t_1\) and \(t_2\) dual to \(s=(1,0)\). Right: the set \(\Wulff\) with the unique vector \(t\) dual to \(s=\frac{1}{\sqrt{5}}(2,1)\).}
	\label{fig:duality}
\end{figure}

The \emph{surcharge function}\footnote{To avoid confusion, we warn the reader that the surcharge function used in~\cite{Aoun+Ott+Velenik-2021} was associated to the inverse correlation, not the interaction.} associated to a dual vector \(t\in\partial\Wulff\) is then defined by
\begin{equation*}
	\surcharge_t(x) = \rho(x)- x\cdot t.
\end{equation*}
It immediately follows from the definition that \(\surcharge_t(x)\geq 0\) for all \(x\in\bbR^d\) and \(\surcharge_t(s)=0\) if \(t\) is dual to \(s\).

\pagebreak

\subsection{Transition points}

\subsubsection{The phase transition}

There are \textit{a priori} two natural transition points in the Random-Cluster model. The first one corresponds to the onset of percolation and reduces to the usual order/disorder phase transition in the associated Potts model when \(q\geq 2\) is an integer:
\[
	\betac(d,q) = \inf\setof{\beta\geq 0}{\Phi^0_{\beta}(0\leftrightarrow\infty) > 0}.
\]
The second one corresponds to the boundary of the regime in which the
connection probabilities decay exponentially fast with the distance, and this uniformly over boundary conditions:
\[
\betaexp(d,q) = \sup\setof{\beta\geq 0}{\exists c>0, \exists N_0\geq 0,\forall N\geq N_0, \Phi^1_{\Lambda_N;\beta}(0\leftrightarrow \Lambda_N^\comp) \leq \sfe^{-cN}}.
\]
We will see below that the transition is \emph{sharp}, that is, these two points actually coincide: \(\betac = \betaexp\).

\subsubsection{Saturation transition}
\label{sec:saturation_transition}

Suppose that \(J\) is exponentially-bounded. The \emph{inverse correlation length} is defined as follows: for \(s\in\bbS^{d-1}\),
\begin{equation}
	\nu_\beta(s) = -\lim_{n\to\infty} \frac1{n}\log G_\beta(0,ns),
\end{equation}
where integer parts are implicitly taken on \(ns\).
The limit can be proven to exist (see~\cite{Aoun+Ioffe+Ott+Velenik-CMP-2021} for references). One always has
\begin{equation*}
G_\beta(0,ns)\leq e^{-\nu_{\beta}(ns)}.
\end{equation*}

 Moreover, when \(\beta<\betaexp\), \(\nu_\beta\) can be extended to a non-degenerate norm on \(\Rd\) by positive homogeneity of order one.
In addition, the function \(\beta\mapsto \nu_\beta(s)\) is non-increasing over \(\R_+\), \(\nu_\beta(s)\leq\rho(s)\).

\medskip
When \(J\) is exponentially-decaying, we define the \emph{saturation point} as
\[
	\betasat(s) = \inf\setof{\beta\geq 0}{\nu_{\beta}(s) < \rho(s)}.
\]
Recall that \(\lim_{\beta\to 0} \nu_\beta(s) = \rho(s)\) (see~\cite{Aoun+Ioffe+Ott+Velenik-CMP-2021}). By definition, one clearly has $\betasat(s)\leq\betaexp$, but a priori, we don't know whether $\betasat(s)<\betaexp$ or not. There is another (family of) point(s) with special properties related to the saturation transition: \(\betahat(s) \equiv \betahat(s,q) \). To define it, introduce the generating functions:
\begin{equation*}
	\bbG_{\beta}(h)=\sum_{x\in\mathbb{Z}^{d}}e^{h\cdot x}G_\beta(0,x)
	\qquad\text{and}\qquad
	\bbJ(h)=\sum_{x\in\mathbb{Z}^{d}}e^{h\cdot x}J_{0,x},
\end{equation*}
for $h \in \R^d$. The closure of the convergence domain of $\bbJ$ is $\Wulff$. 
The point \(\betahat(s)\) is then defined by
\[
	\betahat(s) = \sup_{\substack{t\in\partial\Wulff \\ t\text{ dual to }s}} \sup\setof{\beta\geq 0}{\bbG_{\beta}(t) < \infty}.
\]
We believe that \(\betahat(s) = \betasat(s)\) (see Conjecture~\ref{conj:betasat_equal_betasathat}). Partial results in that direction are proven in the present work. Note that one always has \(\betahat(s)\leq \betasat(s)\) (as the convergence of \(\bbG\) implies the convergence of \(\bbJ\)).

\medskip
Finally, when\footnote{We write \(\psi(x)\propto \rho(x)^{-\alpha}\), since a multiplicative constant is needed in order to ensure that \(\sum_{x\in\Zd} J_x = 1\).} $\psi(x)\propto \rho(x)^{-\alpha}$ with $\alpha\in\R_{>0}$, we define, for any \(s\in\mathbb{S}^{d-1}\),
\[
	\alphasat(s)=\sup\setof{\alpha\in\mathbb{R}_{>0}}{\betasat(s,\alpha)=0}.
\]
Using Theorem~\ref{thm:main_saturation_citerium}, and \(\surcharge_t\geq 0\), it is easy to see that one always has $d\geq \alphasat(s)\geq 1$.

\section{Results}\label{sec:Results}

\subsection{Sharpness and mixing in the Random-Cluster model}
Our first result establishes coincidence of the two transition points \(\betac\) and \(\betaexp\) in the case of exponentially bounded interactions. It is an extension of~\cite{Duminil-Copin+Raoufi+Tassion-2017} to the infinite-range setup. This proves Conjecture~1.11 in~\cite{Aoun+Ioffe+Ott+Velenik-CMP-2021} for the Random-Cluster model (and thus for Bernoulli percolation and the Ising and Potts models). Weaker versions of sharpness were known in this setup for \(q=1\) and \(q=2\)~\cite{Aizenman+Barsky-1987, Aizenman+Barsky+Fernandez-1987} (finite susceptibility rather than exponential decay of connectivities uniformly in boundary conditions).
\begin{theorem}
	\label{thm:main_sharpness}
	Assume that \(J\) is exponentially-bounded, \(d\geq 1\) and \(q\geq 1\). Then \(\betac(d,q) = \betaexp(d,q)\).
\end{theorem}
As a corollary of this, we obtain the following mixing property below \(\betac\):
\begin{corollary}
	\label{cor:main_mixing}
	Suppose \(\beta<\betac\) and let \(\Phi_\beta\) be the unique infinite-volume measure at \(\beta\). Then, there exist \(C<\infty\) and \(c>0\) such that, for any \(F,F'\subset E_d\) and any events \(A\in\calF_{F}, B\in\calF_{F'}\) having positive probability,
	\[
		\Bigl\lvert\frac{\Phi_\beta(A\cap B)}{\Phi_\beta(A)\Phi_\beta(B)} - 1\Bigr\rvert \leq \sum_{x\in V_{F}, y\in V_{F'}} C \sfe^{-c\norm{x-y}}
	\]
	whenever the right-hand-side is at most \(1\).
\end{corollary}

These results are proved in Section~\ref{sec:Sharpness}.

\subsection{Optimal criterion for the existence of a saturation regime}

Our second result removes an unnecessary regularity assumption from the characterization derived in~\cite{Aoun+Ioffe+Ott+Velenik-CMP-2021}, thus answering Open Problem~1.12 therein.

\begin{theorem}
	\label{thm:main_saturation_citerium}
	Suppose \(J\) is exponentially-decaying. Let \(s\in\bbS^{d-1}\). Then, \(\betasat(s) > 0\) if and only if there exists \(t\) \(\rho\)-dual to \(s\) such that \(\bbJ(t) = \sum_{x\in\Zd} \sfe^{t\cdot x} J_x < \infty\).
\end{theorem}
Although the previous result is stated for the Random-Cluster model, it applies to the much more general class of models considered in~\cite{Aoun+Ioffe+Ott+Velenik-CMP-2021}. The proof of Theorem~\ref{thm:main_saturation_citerium} can be found in Appendix~\ref{sec:Criterion}.

\subsection{Sharp asymptotics in the saturation regime}

Our next result provides the sharp asymptotic behavior of the two-point function \(G_\beta(0,ns)\) in the saturation regime \((0, \betasat(s))\). 
It shows, in particular, that these asymptotics are not of Ornstein--Zernike type in this regime.

Let \((\constPrefact_n(s))\) be the following sequence
\[
\constPrefact_n(s) = \constPrefact_n(\beta,q,s) = \frac{\beta}{q} \sfe^{\rho(ns)} \sum_{u,v\in\Zd} \Phi_\beta(0\leftrightarrow u) \sfe^{-\rho(ns-u-v)} \Phi_\beta(0\leftrightarrow v) .
\]

The first claim is a result valid at any \(\beta<\betahat(s)\) (by opposition to what we proved in~\cite{Aoun+Ioffe+Ott+Velenik-CMP-2021} which was at \(\beta\) sufficiently small).
\begin{theorem}\label{theorem:sharp_asymptotics}
	Let \(s\in\bbS^{d-1}\). Suppose that \(\psi\) is of one of the following forms:
	\begin{itemize}
		\item \(\psi(x) \propto \rho(x)^{-\alpha}\) with $\alpha>\alphasat(s)$,
		\item \(\psi(x) \propto \sfe^{-\tilde{c}\rho(x)^{\eta}}\) with \(\tilde{c}>0\) and \(\eta\in (0,1)\).
	\end{itemize}
Then, for every \(\beta<\betahat(s)\), the limit \(\constPrefact(s) = \lim_{n\to \infty} \constPrefact_n(\beta, q,s) \) exists and
	\[
	\Phi_\beta(0\leftrightarrow ns) = \constPrefact J_{0,ns}(1+\sfo_n(1)).
	\]
\end{theorem}
The constant \(\constPrefact(s)\) is reminiscent of the squared susceptibility appearing in the case of sub-exponentially decaying interactions, see Theorem~\ref{thm:asymp_sub_exp_coupling}.

The second claim supplements the first one by giving prefactors for which \(\betasat\) and \(\betahat\) agree (giving a partial answer to Conjecture~\ref{conj:betasat_equal_betasathat} below).
\begin{theorem}\label{theorem:betasat=betahat}
Let \(s\in\bbS^{d-1}\). Suppose that \(\psi\) is of one of the following forms:
	\begin{itemize}
		\item \(\psi(x) \propto \rho(x)^{-\alpha}\) with $\alpha>2d$,
		\item \(\psi(x) \propto \sfe^{-\tilde{c}\rho(x)^{\eta}}\) with \(\tilde{c}>0\) and \(\eta\in (0,1)\).
	\end{itemize}
	Then, $\betasat(s)=\betahat(s)$.
\end{theorem}
This in particular implies that the conclusion of Theorem~\ref{theorem:sharp_asymptotics} holds in the whole saturation regime for a restricted class of prefactors.

\begin{remark}
	\label{rem:main_thm_prefactor_more_general}
	Although, for simplicity of exposition, the results above are only stated for two particular classes of prefactors, they actually hold more generally than that, with the same proof. The main properties of the prefactor are listed in~\eqref{eq:psi_property}. We also need some mild regularity and monotonicity (for instance, \(\psi(x) \asymp f(\rho(x))\) for some nonincreasing positive function \(f\)).
\end{remark}

\begin{conjecture}
	\label{conj:betasat_equal_betasathat}
	$\betasat(s)=\betahat(s)$ for general \(\rho,\psi\).
\end{conjecture}

Finally, as a rather simple adaptation of our methods, we can also obtain sharp asymptotics for some coupling constants decaying slower than exponentially with the distance. Namely, we prove
\begin{theorem}\label{thm:asymp_sub_exp_coupling}
	Let \(J_{0,x} \propto \rho(x)^{-\alpha}\) with \(\alpha>d\), or \(J_{0,x} \propto \sfe^{-\tilde{c}\rho(x)^\eta}\) with \(\tilde{c}>0,\eta\in (0,1)\). Then, for any \(\beta<\betac\),
	\[
		\Phi_\beta(0\leftrightarrow x) = \frac{\beta\chi(\beta)^2}{q} J_{0,x} (1+\sfo_{\norm{x}}(1)),
	\]
	where \(\chi(\beta) = \sum_{y\in\Zd} \Phi_\beta(0\leftrightarrow y)\).
\end{theorem}

The fact that \(\chi(\beta)<\infty\) when \(\beta<\betac\) was proved in~\cite{Hutchcroft-2020}. The case \(q=2\) (Ising model) was treated in~\cite{Newman+Spohn-1998}. In the polynomial case, the generalization of the latter to \(q\geq 1\) was achieved in~\cite{Aoun-2021}. As mentioned in Remark~\ref{rem:main_thm_prefactor_more_general} (in the case of exponentially decaying coupling constants), our methods allow to prove the same claim under more general assumptions.

\begin{conjecture}
	\label{conj:asymptotics_beta>beta_c}
	The result of Theorem~\ref{thm:asymp_sub_exp_coupling} holds for large enough values of \(\beta\) both if we consider $G_{\beta}(0,ns)=\Phi_\beta(0\leftrightarrow ns, \abs{C(0)}<\infty)$ and for the truncated two-point function of the Ising model without an external field. It also holds in presence of an external field for any $\beta>0$.
\end{conjecture}
\begin{remark}
	Note that it is known~\cite{Imbrie+Newman} that the truncated 2-point function of the one-dimensional Ising model with interactions of the form \(J_x \propto \|x\|^{-2}\) and no external field does not display the asymptotic behavior of Theorem~\ref{thm:asymp_sub_exp_coupling} in an intermediate regime of temperatures \(\beta\in [\betac,\beta_0)\) for some finite \(\beta_0>\betac\).
\end{remark}

\subsection{Size of a typical cluster of $0$ and $ns$ in the saturation regime}

Our proof of Theorem~\ref{theorem:sharp_asymptotics} gives a control on the size of $C(0)$ conditionned on $\lbrace 0\leftrightarrow ns\rbrace$ in the saturation regime.

\begin{theorem}\label{theorem:size_cluster}
	Let \(s\in\bbS^{d-1}\). Suppose that \(\psi\) is of one of the following forms:
	\begin{itemize}
		\item \(\psi(x) \propto \rho(x)^{-\alpha}\) with $\alpha>\alphasat(s)$,
		\item \(\psi(x) \propto \sfe^{-\tilde{c}\rho(x)^{\eta}}\) with \(\tilde{c}>0\) and \(\eta\in (0,1)\).
	\end{itemize}
Then, for every \(\beta<\betahat(s)\), there exist \(c>0\) and \(C\) such that, for all $M>0$ and $n > 0$,
	\[
	\Phi_\beta(\abs{C(0)}>M \vert\hphantom{,} 0\leftrightarrow ns)\leq C\sfe^{-c M}.
	\]
\end{theorem}
The previous Theorem is in contrast with what happens for $\beta\in (\betasat(s),\betac)$ (see Fig.~\ref{fig:CondensationPaths}): under suitable assumptions, it was proved for $q=2$ in~\cite{Aoun+Ott+Velenik-2021} that a typical path contributing to the high-temperature expansion of the two-point function has a number of points that is linear in $n$ (see~\cite{Aoun+Ott+Velenik-2021} for a much more precise statement).

\begin{proof}[Proof of Theorem~\ref{theorem:size_cluster}]
Fix $\beta<\beta'<\betasat(s)$. Theorem~\ref{theorem:sharp_asymptotics} implies that for $n$ big enough
\begin{equation}\label{inequality:close_to_prefactor}
\frac{1}{2}\constPrefact(s,\beta')J_{0,ns}\leq \Phi_{\beta'}\bigl(0\leftrightarrow ns \bigr)\leq 2\constPrefact(s,\beta')J_{0,ns}.
\end{equation}
Moreover, it follows from~\eqref{eq:large_deviation_volume_connect} that 
\begin{equation*}
\Phi_\beta\bigl(0\leftrightarrow ns,|C_0|\geq M \bigr) \leq C \Phi_{\beta'}\bigl(0\leftrightarrow ns,|C_0|\geq N \bigr) \sfe^{-cM}
\leq 
2C\constPrefact(s,\beta')J_{0,ns}\sfe^{-cM}.
\end{equation*}
where the last inequality follows from~\eqref{inequality:close_to_prefactor}.
\end{proof}

\subsection{Strict monotonicity of \(\nu\) outside of the saturation regime}

Our last result concerns the regime \((\betasat(s),\betac)\). More precisely, we prove that the function $\beta\mapsto\nu_{\beta}(s)$ is \emph{strictly} decreasing outside of the saturation regime.
\begin{lemma}\label{lemma:bound_on_nu}
	Suppose \(J\) is exponentially-decaying. Let \(s\in\bbS^{d-1}\), and suppose \(\betasat(s)>0\). Then, for any $\beta\in(\betasat(s),\betac)$, there exists $\varepsilon=\varepsilon_{\beta}>0$ and $C=C_{\beta, \varepsilon}>0$ such that, for any $\beta'\in(\beta,\beta+\varepsilon)$, one has
	\begin{equation*}
		\nu_{\beta}(s)-\nu_{\beta'}(s)\geq C(\beta'-\beta).
	\end{equation*}
	In particular, the function $\beta\mapsto\nu_{\beta}(s)$ is strictly decreasing on $(\betasat(s),\betac)$.
\end{lemma}

We believe that \(\beta\mapsto \nu_{\beta}(s)\) is (real-)analytic on the interval \((\betasat(s),\betac)\). We plan to come back to this issue in a future work. As the proof is short and does not fit naturally in the remaining sections, we present it here (although it refers to other parts of the paper).
\begin{proof}[Proof of Lemma~\ref{lemma:bound_on_nu}]
	From Theorem~\ref{thm:main_saturation_citerium}, \(\betasat(s)>0\) implies the existence of \(t\in \partial \Wulff\) dual to \(s\) with \(\bbJ(t)<\infty\). Fix such a \(t\). On the one hand, since \(c_s=\rho(s)- \nu_{\beta}(s)>0\), we have (as \(t\cdot s = \rho(s)\) by choice of \(t\))
	\begin{multline*}
		\Phi_\beta(0\leftrightarrow ns,\abs{C_{0}}\leq \varepsilon n)
		\leq \sum_{\substack{\gamma:0\to x\\|\gamma|\leq \epsilon n}} \Phi_\beta(\gamma \textnormal{ open})\leq \\
		\leq \sfe^{-\rho(ns)}\sum_{k=1}^{\varepsilon n}\sum_{\substack{y_1,\cdots, y_k\in \Zd\\ \sum y_i = x}} \prod_{i=1}^k \sfe^{t\cdot y_i}(1-\sfe^{-\beta J_{y_i}}) \leq \sfe^{-n(\nu_{\beta}(s)+ c_s)} \sum_{k=1}^{\varepsilon n} \Big(\beta\sum_{y\in \Zd}  J_{y}\sfe^{t\cdot y}\Big)^k
	\end{multline*}where the sum after the first inequality is over self avoiding paths \(\gamma=(0, y_1,y_1+y_2,\cdots, x)\), we used finite energy to get the second inequality, and \(1-\sfe^{-x}\leq x\) for \(x\geq 0\) to get the third.
	
	On the other hand, we know by~\eqref{eq:OSSSexpovolume} that 
	\[
	\Phi_\beta(0\leftrightarrow ns,\abs{C_{0}} \geq \varepsilon n)
	\leq 
	C\sfe^{-\nu_{\beta'}(ns)} \sfe^{-c_{\beta,\beta'}\varepsilon n},
	\]
	where \(c_{\beta,\beta'} = \frac{\beta'-\beta}{2\beta'\Phi_{\beta'}(|C_0|)}\). To conclude, introduce \( A = \max(1,\beta\bbJ(t))<\infty\) and \(a=\log A\), and note that combining the two bounds with \(\Phi_\beta(0\leftrightarrow ns)\geq e^{-\nu_{\beta}(ns)(1+\sfo_n(1)) }\) gives (for \(n\) large enough)
	\begin{equation*}
		e^{\sfo(n) } \leq C\sfe^{-(\nu_{\beta'}(s)-\nu_{\beta}(s) + c_{\beta,\beta'}\varepsilon )n} + \varepsilon n \sfe^{-n (c_s- a \varepsilon )}.
	\end{equation*}Taking \(\varepsilon>0\) small enough (as a function of \(s,\beta, \beta'-\beta\)) the last display implies
	\begin{equation*}
		\nu_{\beta'}(s)-\nu_{\beta}(s) + c_{\beta,\beta'}\varepsilon \leq 0,
	\end{equation*}which is the claim.
\end{proof}

\begin{remark}
	The same lower bound can be proved if $(J_{x,y})_{x,y\in\mathbb{Z}^{d}}$ are finite range. Indeed, in this case
	\[
	\Phi_\beta(0\leftrightarrow ns, \abs{C_{0}}\leq \varepsilon n) = 0,
	\]
	for \(\varepsilon\) small enough, and~\eqref{eq:OSSSexpovolume} still holds in this case (the inequality~\eqref{eq:OSSSexpovolume} holds for any translation-invariant coupling constants \((J_{0,x})_{x\in\Zd}\) that satisfy \(\sum_x J_{0,x} < \infty\)).
\end{remark}

\section{Sharpness of the phase transition}\label{sec:Sharpness}

We will use the shorter notation
\[
	\Phi_{N;\beta} \equiv \Phi_{\Lambda_N;\beta}^1.
\]

The goal of this section is to prove the following result, which implies Theorem~\ref{thm:main_sharpness}.
\begin{theorem}
	\label{thm:sharpness_main}
	Suppose \(J\) is exponentially-bounded. Let \(\tilde{\betac}\) be given by~\eqref{eq:betac_tilde} below. Then the following assertions hold.
	\begin{itemize}
		\item For any \(\beta<\tilde{\betac}\), there exist \(C_{\beta}\) and \(c_{\beta}>0\) such that
		\[
			\Phi_{N;\beta}(0\leftrightarrow \Lambda_N^\comp)\leq C_{\beta} \sfe^{-c_{\beta} N},
		\]
		for all \(N\geq 1\).
		\item For any \(\beta>\tilde{\betac}\),
		\[
			\Phi_\beta^1(0\leftrightarrow \infty)
			\geq (\beta-\tilde{\betac}) \frac{\sfe^{-\beta}}{9\beta}.
		\]
	\end{itemize}
	In particular, \(\tilde{\betac}=\betac=\betaexp\).
\end{theorem}

Our approach is closely related to the one introduced in~\cite{Duminil-Copin+Raoufi+Tassion-2017}, which is based on the one-monotonic version of the OSSS inequality. The main (mostly technical) differences will be highlighted during the proofs.

\subsection{Preparations}
We will need a few more objects. We first define the edge sets
\[
	E_{N,R} = \bsetof{\{x,y\}\in\bar{E}_{\Lambda_{N}}}{\normsup{x-y} \leq R},
	\qquad\text{and}\qquad
	E_{N,>R} = E_{N,\infty} \setminus E_{N,R}.
\]
Also, \(E_{\infty,R} = \lim_{N\to \infty} E_{N,R}\). Define the shorthand \(\{x\leftrightarrow_R y\} = \{x\xlongleftrightarrow{E_{\infty,R}} y\}\).

We will use the shorter notation
\[
	f_N(\beta) = \Phi_{2N;\beta}(0\leftrightarrow \Lambda_N^\comp),
	\quad
	F_{N}(\beta) = \sum_{k=0}^{N-1} f_{k}(\beta).
\]

We then define
\begin{align}\label{eq:betac_tilde}
	\tilde{\betac} &= \sup\setof{\beta\geq 0}{\exists (N_n)_{n\geq 1} \text{ increasing}, \exists C, c>0, F_{N_n}(\beta) \leq C(N_n)^{1-c}}\notag\\
	&= \inf\Bsetof{\beta\geq 0}{\liminf_{N\to\infty} \frac{\log F_N(\beta)}{\log N}\geq 1}.
\end{align}
Notice that, for any \(\beta<\tilde{\betac}\), as \(f_N(\beta)\) is non-increasing in \(N\), one has \(\lim_{N\to\infty} f_N(\beta) = 0\). In particular, \(\tilde{\betac}\leq \betac\).
A first difference compared to~\cite{Duminil-Copin+Raoufi+Tassion-2017} is the use of a \(\liminf\) instead of a \(\limsup\) in the definition of \(\tilde{\betac}\). This will be convenient when establishing that \(\tilde{\betac}\geq \betac\), but will generate some difficulties when proving that \(\beta<\tilde{\betac}\).

\subsection{Differential inequality: radius}
 
We will use the following differential inequality.
\begin{lemma}\label{lem:diff_ineq_radius}
	For any \(\beta>0\), \(\infty\geq R'\geq R>0\), and \(N\geq 1\),
	\begin{equation}
		\frac{\dd}{\dd\beta}\Phi_{2N;\beta}(0\leftrightarrow_{R'} \Lambda_N^\comp) \geq \frac{\sfe^{-\beta}}{\beta} \Phi_{2N;\beta}(0\leftrightarrow_R \Lambda_N^\comp) \frac{N}{R+4\sum_{i=0}^{N-1} \Phi_{N;\beta}(0\leftrightarrow_R \Lambda_{i}^\comp)}.
	\end{equation}
\end{lemma}
\begin{proof}
	First,
	\begin{align}
			\frac{\dd}{\dd\beta}\Phi_{2N;\beta}(0\leftrightarrow_{R'} \Lambda_N^\comp)
			&= \sum_{e\in E_{2N,\infty}} \frac{J_e}{1-\sfe^{-\beta J_{e}}} \Phi_{2N;\beta}(0\leftrightarrow_{R'} \Lambda_N^\comp\ ;\ \omega_e) \notag\\
			&\geq \frac1{\beta} \sum_{e\in E_{2N,R}} \Phi_{2N;\beta}(0\leftrightarrow_{R'} \Lambda_N^\comp\ ;\ \omega_e),
			\label{eq:diff_ineq_prf1}
	\end{align}
	since \(\frac{x}{1-\sfe^{-\beta x}} \geq \beta^{-1}\) and the covariances are nonnegative by FKG.
	
	We will use the two-function version of the monotonic OSSS inequality of~\cite{Duminil-Copin+Raoufi+Tassion-2017}, see~\cite[Theorem 2.2]{Hutchcroft-2020} for the exact statement. We refer to~\cite{Duminil-Copin+Raoufi+Tassion-2017, Hutchcroft-2020} for missing definitions; our notations should be close enough to the ones used in theses papers for the reader to be able to translate. We will use the following inputs in~\cite[Theorem 2.2]{Hutchcroft-2020}:
	\begin{gather*}
		\mu = \Phi_{2N;\beta}, \quad
		f = \mathds{1}_{0\leftrightarrow_{R'} \Lambda_N^\comp}, \quad
		g = \mathds{1}_{0\leftrightarrow_R \Lambda_N^\comp}.
	\end{gather*}
	We obtain that, for any decision tree \(T\) computing \(g\),
	\begin{equation}\label{eq:OSSS_DCRT}
		\Phi_{2N;\beta}(f\ ;\ g) \leq \sum_{e\in E_{2N,\infty}} \delta_e(\Phi_{2N;\beta},T) \Phi_{2N;\beta}(\omega_e\ ;\ f).
	\end{equation}
	We will now define some decision trees. Notice that \(\mathds{1}_{0\leftrightarrow_R \Lambda_N^\comp}\) is measurable with respect to \(\calF_{E_{N,R}}\). Fix some arbitrary total ordering of \(E_{N,R}\). We define a family of decision trees \(T^i, i=1,\dots, N\), as follows. \(T^{i}\) first queries the state of all edges \(\{x,y\}\in E_{N,R}\) with \([x,y]\cap \partial [-i,i]^d\neq \varnothing\) in increasing order. Let us denote the set of open edges revealed in this way by \(X\). Then, \(T^{i}\) explores all the connected components, in the configuration restricted to \(E_{N,R}\), of the endpoints of \(X\) together with their boundary. We refer to~\cite{Duminil-Copin+Raoufi+Tassion-2017,Hutchcroft-2020} for an explicit description of the exploration algorithm.
	Obviously, \(T^i\) computes \(\mathds{1}_{0\leftrightarrow_R \Lambda_N^\comp}\). Moreover, for an edge to be queried, it has to be in \(E_{N,R}\) and either intersect \(\partial [-i,i]^d\) or be connected to an open edge which does so. The revealment of an edge in \(E_{N,R}\) is therefore upper bounded by
	\begin{multline*}
		\delta_{\{x,y\}}(\Phi_{2N;\beta},T^{i}) \leq\\
		\mathds{1}_{[x,y]\cap \partial [-i,i]^d\neq \varnothing} + \Phi_{2N;\beta}(x\leftrightarrow_R \Lambda_{|\normsup{x}-i|}(x)^\comp) + \Phi_{2N;\beta}(y\leftrightarrow_R \Lambda_{|\normsup{y}-i|}(y)^\comp).
	\end{multline*}
	Now,
	\[
		\sum_{i=1}^N \Phi_{2N;\beta}(x\leftrightarrow_R \Lambda_{|\normsup{x}-i|}(x)^\comp)
		\leq \sum_{i=1}^N\Phi_{N;\beta}(0\leftrightarrow_R \Lambda_{|\normsup{x}-i|}^\comp)
		\leq 2 \sum_{i=0}^{N-1} \Phi_{N;\beta}(0\leftrightarrow_R \Lambda_{i}^\comp).
	\]
	Taking the average over \(i=1,\dots,N\), one gets
	\[
		\frac1{N}\sum_{i=1}^N \delta_{\{x,y\}}(\Phi_{2N;\beta}, T^i)
		\leq \frac{\normsup{x-y}}{N} + \frac{4}{N} \sum_{i=0}^{N-1} \Phi_{N;\beta}(0\leftrightarrow_R \Lambda_{i}^\comp).
	\]
	In particular, averaging~\eqref{eq:OSSS_DCRT} over \(T^i\), \(i=1,\dots, N\), implies
	\begin{multline}\label{eq:diff_ineq_prf2}
		\sum_{e\in E_{2N,R}} \Phi_{2N;\beta}(\omega_e\ ;\ 0\leftrightarrow_{R'} \Lambda_N^\comp) \geq \sum_{e\in E_{N,R}} \Phi_{2N;\beta}(\omega_e\ ;\ 0\leftrightarrow_{R'} \Lambda_N^\comp)\geq\\
		\Phi_{2N;\beta}(0\leftrightarrow_{R'} \Lambda_N^\comp\ ;\ 0\leftrightarrow_R \Lambda_N^\comp)\frac{N}{R+4\sum_{i=0}^{N-1} \Phi_{N;\beta}(0\leftrightarrow_R \Lambda_{i}^\comp)}.
	\end{multline}
	Finally,
	\begin{multline}\label{eq:diff_ineq_prf3}
		\Phi_{2N;\beta}(0\leftrightarrow_{R'} \Lambda_N^\comp\ ;\ 0\leftrightarrow_R \Lambda_N^\comp)
		=
		\Phi_{2N;\beta}(0\leftrightarrow_R \Lambda_N^\comp)\Phi_{2N;\beta}(0\nleftrightarrow_{R'} \Lambda_N^\comp) \\
		\geq \Phi_{2N;\beta}(0\leftrightarrow_R \Lambda_N^\comp)\Phi_{2N;\beta}(0\nleftrightarrow \Lambda_N^\comp) \geq \Phi_{2N;\beta}(0\leftrightarrow_R \Lambda_N^\comp) \sfe^{-\beta},
	\end{multline}
	where the last inequality follows from finite energy and the normalization \(\sum_x J_{0x} = 1\). Using~\eqref{eq:diff_ineq_prf1}, \eqref{eq:diff_ineq_prf2} and~\eqref{eq:diff_ineq_prf3} yields the result.
\end{proof}

We will combine the previous differential inequality with a simple bound comparing \(\Phi_{2N;\beta}(0\leftrightarrow_R \Lambda_N^\comp)\) and \(\Phi_{2N;\beta}(0\leftrightarrow \Lambda_N^\comp)\). Partitioning on whether there is an edge of length at least \(R\) which is open or not, and using a union bound, one has
\begin{align*}
	\Phi_{2N;\beta}(0\leftrightarrow \Lambda_N^\comp) &\leq \Phi_{2N;\beta}(0\leftrightarrow_R \Lambda_N^\comp) + \sum_{e\in E_{2N,>R}} \Phi_{2N;\beta}(\omega_e=1)\\
	&\leq \Phi_{2N;\beta}(0\leftrightarrow_R \Lambda_N^\comp) + C\beta N^d \sum_{y: \normsup{y}> R} J_y\\
	&\leq \Phi_{2N;\beta}(0\leftrightarrow_R \Lambda_N^\comp) + C\beta N^d \sfe^{-cR},
\end{align*}
where \(c>0\), we used that \(J\) is exponentially bounded in the last line, and all \(\beta\)-dependencies are explicit.
So, there exist \(C<\infty, c>0\) (independent of \(\beta\)) such that for any \(R>0\)
\begin{equation}
	\label{eq:connections_to_R_connections}
	\Phi_{2N;\beta}(0\leftrightarrow_R \Lambda_N^\comp) \geq \Phi_{2N;\beta}(0\leftrightarrow \Lambda_N^\comp) -C\beta N^d \sfe^{-cR}. 
\end{equation}

\subsection{Proof of Theorem~\ref{thm:sharpness_main}: Percolation above \(\tilde{\betac}\)}

We now study the differential inequality of Lemma~\ref{lem:diff_ineq_radius}. The proof is very close to the corresponding ones in~\cite{Duminil-Copin+Raoufi+Tassion-2017,Hutchcroft-2020}. The main difference is that we have an effect due to the infinite range of the interaction in the differential inequality. This is where using a \(\liminf\) in the definition of \(\tilde{\betac}\) instead of a \(\limsup\) comes crucially into play.

\begin{lemma}
	For any \(\beta>\tilde{\betac}\),
	\[
		\Phi^1_{\beta}(0\leftrightarrow \infty) \geq (\beta-\tilde{\betac})\frac{\sfe^{-\beta}}{9\beta}.
	\]
	In particular, \(\tilde{\betac} \geq \betac\).
\end{lemma}
\begin{proof}
	The claim will follow by lower bounding the quantity \(W_{N,M}\) defined by
	\[
		W_{N,M}(\beta) = \frac1{\log N} \sum_{k=M}^{N} k^{-1}f_{k}(\beta).
	\]Observe that \(W_{N,M}\) is non-decreasing in \(\beta\). The interest of this quantity is that one has \linebreak \(\lim_{M\to\infty}\lim_{N\to\infty} W_{N,M}(\beta) = \Phi^1_{\beta}(0\leftrightarrow \infty)\). Indeed: the lower bound is obtained by observing that \(f_{k}(\beta)\geq \Phi^1_{\beta}(0\leftrightarrow \infty)\) and \(\lim_{N\to \infty}\frac{\sum_{k=M}^N k^{-1}}{\log(N)} =1\), while the upper bound follows from \(f_{k}(\beta) \leq \Phi_{2k;\beta}(0\leftrightarrow \Lambda_M^c)\) for \(k\geq M\), and \(\Phi_{2k;\beta}(0\leftrightarrow \Lambda_M^c) \leq \Phi_{\beta}^1(0\leftrightarrow \Lambda_M^c) +\sfo_k(1)\) (by convergence of \(\Phi_{2k;\beta}\) to \(\Phi_{\beta}^1\)), so that \(\lim_{N\to \infty} W_{N,M} \leq \Phi_{\beta}^1(0\leftrightarrow \Lambda_M^c)\).
	
	Let \(\beta>\beta'>\tilde{\betac}\). Let \(R_N = (d+\frac{3}{2})\frac{\log N}{c}\) where \(c\) is given by~\eqref{eq:connections_to_R_connections}, and let \(N_0\) be such that \(R_N\leq N^{1/2}\) for all \(N\geq N_0\).
	Using the bound
	\begin{equation}\label{eq:perco_prf0}
		\sum_{i=0}^{N-1} \Phi_{N;\beta}(0\leftrightarrow_R \Lambda_{i}^\comp)
		\leq 2 \sum_{i=0}^{N/2} \Phi_{2i;\beta}(0\leftrightarrow \Lambda_{i}^\comp)
		\leq 2 F_{N}(\beta)
	\end{equation}
	in Lemma~\ref{lem:diff_ineq_radius} (with \(R'=\infty\)), and using~\eqref{eq:connections_to_R_connections}, one obtains
	\begin{equation}\label{eq:perco_prf1}
		\frac{\rmd}{\rmd\beta}f_N(\beta) \geq \frac{\sfe^{-\beta}}{\beta} \Phi_{2N;\beta}(0\leftrightarrow_{R_N} \Lambda_N^c ) \frac{N}{R_N+8F_N(\beta)}.
	\end{equation}
	Recall that, by the definition of \(\tilde{\betac}\) in~\eqref{eq:betac_tilde}, one has
	\begin{equation}\label{eq:perco_prf2}
		\liminf_{N\to\infty}\frac{\log F_N(\beta')}{\log N} \geq 1.
	\end{equation}
	In particular, by monotonicity of \(F_{N}(\beta)\) in \(\beta\), there exists \(N_1 = N_1(\beta') < \infty\) such that, for any \(N\geq N_1\) and \(\beta\geq\beta'\), \(F_N(\beta) \geq N^{1/2}\). Let \(N_2 = \max(N_0, N_1)\).
	
	By~\eqref{eq:perco_prf1}, our choice of \(N_2\), and the inequality \(R_k\leq k^{1/2}\) for \(k\geq N_0\), for any \(M\geq N_2\)
	\begin{align*}
		\frac{\rmd}{\rmd\beta} W_{N,M}(\beta)
		&\geq \frac{\sfe^{-\beta}}{9\beta} \frac1{\log N} \sum_{k=M}^{N}  \frac{\Phi_{2k;\beta}(0\leftrightarrow_{R_k} \Lambda_k^c )}{F_k(\beta)}\\
		&\geq \frac{\sfe^{-\beta}}{9\beta} \frac1{\log N} \Bigl(\log F_{N+1}(\beta) - \log F_{M}(\beta) - \sum_{k=M}^{N}  \frac{C\beta }{k^{2}}\Bigr),
	\end{align*}
	where we used~\eqref{eq:connections_to_R_connections}, \(F_k(\beta)\geq k^{1/2}\), the choice of \(R_k\), and
	\[
		\frac{f_k}{F_k} = \frac{F_{k+1}-F_k}{F_k} \geq \int_{F_k}^{F_{k+1}} \frac1{x} \,\rmd x = \log F_{k+1} - \log F_k.
	\]
	Integrating this inequality between \(\beta'\) and \(\beta\) and using the monotonicity of $F_{k}$ in $\beta$ yields, for \(M\geq N_2\),
	\[
		W_{N,M}(\beta) \geq (\beta-\beta') \frac{\sfe^{-\beta}}{9\beta} \frac1{\log N} \Bigl(\log F_{N+1}(\beta') - \log F_{M}(\beta) - c(\beta,\beta')M^{-1}\Bigr).
	\]
	Taking \(N\to\infty\) followed by \(M\to\infty\) and using~\eqref{eq:perco_prf2}, one obtains
	\[
		\Phi^1_{\beta}(0\leftrightarrow \infty) \geq (\beta-\beta') \frac{\sfe^{-\beta}}{9\beta}.
	\]
	The result now follows by letting \(\beta'\downarrow\tilde{\betac}\).
\end{proof}

\subsection{Proof of Theorem~\ref{thm:sharpness_main}: Exponential decay below \(\tilde{\betac}\)}

We will rely on two elementary lemmas on sequences, the proofs of which are relegated to the end of the section.
\begin{lemma}\label{lem:stretch_exp_seq}
	Let \(a_N\geq 0\) be a sequence satisfying
	\begin{itemize}
		\item \(a_N\) is non-increasing,
		\item \(\exists m\in\bbZ_{\geq 2}\), \(\exists\alpha<\infty\), \(\exists C_1,C_2\geq 0, \exists c_1>0, \forall N\geq 1, a_{m N} \leq C_1 N^{\alpha} a_N^2 + C_2\sfe^{-c_1 N}\),
		\item \(\exists\epsilon>0, \exists (N_{n})_{n\geq 1}\) increasing, \(\forall n\geq 1, a_{N_n} \leq \sfe^{- (\log N_n)^{1+\epsilon}}\).
	\end{itemize}
	Then there exist \(C\geq 0\) and \(c>0\) such that
	\[
		\forall N\geq 1, \quad a_N\leq C\sfe^{-c N^{\nu}},
	\]
	where \(\nu = \frac{\log 2}{\log m}\).
\end{lemma}

\begin{lemma}\label{lem:stretch_exp_to_exp_seq}
	Let \(a_N\geq 0\) be a sequence satisfying
	\begin{itemize}
		\item \(\exists c_1>0,\exists\alpha<\infty,\exists \tilde{N}, \forall N\geq \tilde{N}\),
		\[
			a_{N}\leq \sfe^{-c_1 N} + N^{\alpha} \sum_{k=\lceil N/3\rceil }^{\lceil 2N/3\rceil} a_{k}a_{N-k},
		\]
		\item \(\exists\epsilon>0, \exists\tilde{N}, \forall N\geq\tilde{N}, a_N\leq \sfe^{-N^{\epsilon}}\).
	\end{itemize}
	Then, there exist \(C\geq 0\) and \(c>0\) such that
	\[
		\forall N\geq 1, \quad a_N\leq C\sfe^{-c N}.
	\]
\end{lemma}

The idea is to establish the stretch-exponential decay of \(f_N(\beta)\) along a subsequence (provided by the definition of \(\tilde{\betac}\)) by integrating the \(\log\) version of the differential inequality from Lemma~\ref{lem:diff_ineq_radius}.
We then use Lemma~\ref{lem:stretch_exp_seq} to push this result to all \(N\geq 1\).
Finally, we use Lemma~\ref{lem:stretch_exp_to_exp_seq} to enhance the stretch-exponential decay to exponential decay.
This last step also differs from the argument in~\cite{Duminil-Copin+Raoufi+Tassion-2017}, which uses a second integration of the differential inequality. The differential inequality we obtain is not very convenient to repeat the argument of~\cite{Duminil-Copin+Raoufi+Tassion-2017}.

\begin{lemma}\label{lem:exp_dec}
	For any \(\beta<\tilde{\betac}\), there exist \(C_{\beta}\geq 0\) and \(c_{\beta}>0\) such that, for any \(N\geq 1\),
	\[
		\Phi_{N;\beta}(0\leftrightarrow \Lambda_{N}^\comp) \leq C_{\beta} \sfe^{-c_{\beta} N}.
	\]
\end{lemma}
\begin{proof}
	Let \(\beta'<\tilde{\betac}\). By definition of \(\tilde{\betac}\), there exist \(\tilde{c}=\tilde{c}(\beta')>0\) and an increasing sequence \((N_n)_{n\geq 1}\) such that \(F_{N_n}(\beta')\leq N_n^{1-\tilde{c}}\) for all \(n\geq 1\). Without loss of generality, we impose \(\tilde{c}<1\). Let \(R_N = N^{1-\tilde{c}}\). The first step consists in bounding \(f_N(\beta)\) from above along the subsequence \((N_n)\).
	\begin{claim}
		\label{proof_of_exp_dec:claim1}
		There exists \(n_0\geq 1\) and \(c''>0\) such that for any \(\beta<\beta'\), there exists \(c'=c'(\beta,\beta')>0\) satisfying
		\begin{equation*}
			f_{N_n}(\beta) \leq \exp(-c' N_n^{c''})
		\end{equation*}for any \(n\geq n_0\).
	\end{claim}
	\begin{proof}
		First, using~\eqref{eq:connections_to_R_connections},
		\begin{equation*}
			f_{N}(\beta) \leq \Phi_{2N;\beta}(0\leftrightarrow_{R_N} \Lambda_N^{\comp}) + C\beta N^d e^{-N^{1-\tilde{c}}}.
		\end{equation*}The second term has the wanted decay, so we focus on the first one. Using Lemma~\ref{lem:diff_ineq_radius} with \(R=R'=R_N\) (and~\eqref{eq:perco_prf0}), one then gets that, for any \(\beta\leq \beta'\)
		\[
		\frac{\rmd}{\rmd\beta} \log \Phi_{2N;\beta}(0\leftrightarrow_{R_N} \Lambda_N^{\comp}) \geq 	\frac{\sfe^{-\beta'}}{\beta'}\frac{N}{R_N+8F_N(\beta')}.
		\]So, for \(n\geq 1\) and \(\beta\leq\beta'\),
		\[
		\frac{\rmd}{\rmd\beta} \log \Phi_{2N_n;\beta}(0\leftrightarrow_{R_{N_n}} \Lambda_{N_n}^{\comp}) \geq \frac{\sfe^{-\beta'}}{\beta'}\frac{N_n}{9 N_n^{1-\tilde{c}}} \equiv c'N_n^{\tilde{c}},
		\]
		where \(c'\equiv c'(\beta')\). Integrating between \(\beta\) and \(\beta'\),
		\[
		\Phi_{2N_n;\beta}(0\leftrightarrow_{R_{N_n}} \Lambda_{N_n}^{\comp}) \leq \Phi_{2N_n;\beta'}(0\leftrightarrow_{R_{N_n}} \Lambda_{N_n}^{\comp}) \exp(-c'(\beta'-\beta) N_n^{\tilde{c}}).
		\qedhere
		\]
	\end{proof}

	The second step is too push Claim~\ref{proof_of_exp_dec:claim1} to all values of \(N\) (not just the subsequence \((N_n)\)). This step is the price to pay for having a \(\liminf\) instead of a \(\limsup\) in~\eqref{eq:betac_tilde}.
	\begin{claim}
		\label{proof_of_exp_dec:claim2}
		For any \(\beta<\beta'\), there exists \(c'>0\) such that for any \(N\) large enough
		\begin{equation*}
			f_{N}(\beta) \leq C'\exp(-c' N^{\log(2)/\log(6)}).
		\end{equation*}
	\end{claim}
	\begin{proof}
		We would now like to use Lemma~\ref{lem:stretch_exp_seq}. The sequence \(a_N=f_{N}(\beta)\) is non-increasing and Claim~\ref{proof_of_exp_dec:claim1} implies that \(a_N\) satisfies the third condition of Lemma~\ref{lem:stretch_exp_seq}. We now establish the second condition with \(m=6\). 
		Partitioning according to whether there is an open edge in \(E_{12N,N}\) or not, we obtain from finite energy that
		\[
		f_{6N}(\beta) \leq \Phi_{12N;\beta}(0\leftrightarrow_{N} \Lambda_{6N}^\comp) + C_{\beta} \sfe^{-cN}
		\]
		for some \(c>0\) and \(C_{\beta}\geq 0\). Now, the event \(\{0\leftrightarrow_{N} \Lambda_{6N}^\comp\}\) implies both the event \(\{0\leftrightarrow_N \Lambda_N^\comp\}\), which is \(E_{2N,N}\)-measurable, and the existence of a point \(x\in\Lambda_{7N}\setminus\Lambda_{6N}\) such that \(x\leftrightarrow_{N} \Lambda_{N}(x)\), which is (\(x+E_{2N,N}\))-measurable. In particular, by a union bound and monotonicity,
		\[
		\Phi_{12N;\beta}(0\leftrightarrow_{N} \Lambda_{6N}^\comp)
		\leq C_d N^{d} \Phi_{2N}(0\leftrightarrow \Lambda_N^\comp)^2.
		\]
		Lemma~\ref{lem:stretch_exp_seq} then implies the existence of \(c=c(\beta)>0\) and \(C=C(\beta)\geq 0\) such that
		\[
		\forall N\geq 1, \quad f_{N}(\beta)\leq C \sfe^{-c N^{\nu}}
		\]
		with \(\nu = \frac{\log 2}{\log 6}\).
	\end{proof}
	The final step is to enhance the stretch exponential decay to exponential. This is the content of the last claim.
	\begin{claim}
		\label{proof_of_exp_dec:claim3}
		Let \(\beta<\beta'\). Then, the sequence
		\[
			b_N = \Phi_{N;\beta}(0\leftrightarrow \Lambda_N^\comp)
		\]
		satisfies the hypotheses of Lemma~\ref{lem:stretch_exp_to_exp_seq}.
	\end{claim}
	\begin{proof}
		The second condition of Lemma~\ref{lem:stretch_exp_to_exp_seq} with \(\epsilon = \frac{\log(2)}{2\log(6)}\) follows from Claim~\ref{proof_of_exp_dec:claim2}. Let us now turn to the first condition.
		 Partitioning on whether an edge with length at least \(N/3\) is open or not, one obtains, for any \(N\) large enough,
		\[
		\Phi_{N;\beta}(0\leftrightarrow \Lambda_N^\comp) \leq \sfe^{-c_1 N} + \Phi_{N;\beta}(0\leftrightarrow_{N/3} \Lambda_N^\comp).
		\]
		Now, the event \(\{0\leftrightarrow_{N/3} \Lambda_N^\comp\}\) entails the existence of \(x,y,z\in\Lambda_N\) (see Fig.~\ref{figure:Claim_3}) with
		\begin{itemize}
			\item \(N/3\leq\normsup{y}\leq\normsup{x}\leq 2N/3\), \(\normsup{z}> \normsup{x}\),
			\item \(0\xlongleftrightarrow{E_{\Lambda_{\normsup{x}}}}_{N/3} x \),
			\item \(\omega_{yz} = 1\),
			\item \(z\xlongleftrightarrow{E_{N,\infty}\setminus E_{\normsup{x},\infty}} \Lambda_N^\comp\).
		\end{itemize}
		
		To see this, let \(\gamma = (\gamma_1,\dots,\gamma_m)\) be a self-avoiding path of open edges in \(E_{N,N/3}\) with \(\gamma_1=0\) and \(\gamma_m\in\Lambda_N^\comp\). Let \(t_1\) be the first time \(\gamma\) exits \(\Lambda_{N/3}\). Set \(x=\gamma_{t_1}\in\Lambda_{2N/3}\setminus \Lambda_{N/3}\). \(x\) connected to \(0\) using only edges in \(E_{\Lambda_{\normsup{x}}}\), since \(\{\gamma_1,\dots,\gamma_{t_1}\}\subset \Lambda_{\normsup{x}}\). Let now \(t_2\) be the last time \(\gamma\) exits \(\Lambda_{\normsup{x}}\). This implies, in particular, that \(\{\gamma_{t_2},\dots,\gamma_{m}\}\subset \Lambda_{\normsup{x}}^\comp\) and \(\gamma_{t_2-1}\in\Lambda_{\normsup{x}}\). We can then set \(y=\gamma_{t_2-1}\) and \(z=\gamma_{t_2}\).
		
		\begin{figure}[t]
	\centering
	\includegraphics{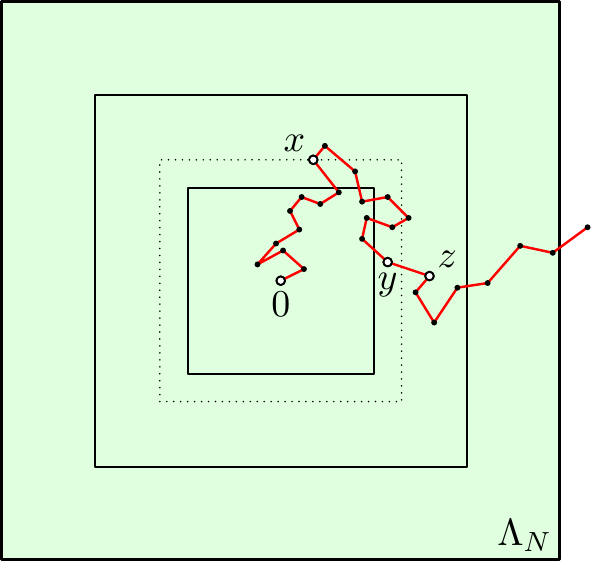}
	\caption{The construction of a triplet $x,y,z$ used in the proof of Claim 3.}
	\label{figure:Claim_3}
\end{figure}
		
		Now, for such a triplet \(x,y,z\),
		\begin{align*}
			\Phi_{N;\beta}&\bigl(0\xlongleftrightarrow{E_{\Lambda_{\normsup{x}}}}_{N/3} x,\omega_{yz} = 1, z\xlongleftrightarrow{E_{N,\infty}\setminus (E_{\normsup{x},\infty} \cup \{\{y,z\}\})} \Lambda_N^\comp \bigr) \\
			&\leq p_{yz} \Phi_{N;\beta}\bigl(0\xlongleftrightarrow{E_{\Lambda_{\normsup{x}}}}_{N/3} x \bgiven z\xleftrightarrow[]{E_{N,\infty}\setminus (E_{\normsup{x},\infty} \cup \{\{y,z\}\}) } \Lambda_N^\comp \bigr) \Phi_{N;\beta}\bigl(z\xlongleftrightarrow{E_{N,\infty}\setminus\{y,z\}} \Lambda_N^\comp \bigr)\\
			&\leq \frac{p_{yz}(p_{yz}+q(1-p_{yz}))}{p_{yz}}\Phi_{\normsup{x};\beta} \bigl(0\xlongleftrightarrow{E_{\Lambda_{\normsup{x}}}} x\bigr) \Phi_{N;\beta}\bigl( y\leftrightarrow \Lambda_N^\comp \bigr)\\
			&\leq c \Phi_{\normsup{x};\beta}\bigl(0\leftrightarrow \Lambda_{\normsup{x}}^\comp\bigr) \Phi_{N-\normsup{y};\beta}\bigl( 0\leftrightarrow \Lambda_{N-\normsup{y}}^\comp \bigr),
		\end{align*}
		where \(p_{xy} = 1-e^{-\beta J_{xy}}\), we opened $\lbrace y,z\rbrace$ in the third line and we forced a step from \(x\) to the outside of \(\Lambda_{\normsup{x}}\) by finite energy (manifested by the presence of the constant \(c<\infty\) depending on \(q\), \(\beta\) and \(J\)).
		Now, by a union bound and monotonicity,
		\begin{align*}
			\Phi_{N;\beta}(0\leftrightarrow_{N/3}& \Lambda_N^\comp) \\
			&\leq\sum_{k=N/3}^{2N/3}\sum_{k'=N/3}^{k}\sum_{\normsup{x}=k}\sum_{\normsup{y}=k'} CN^{d} cq \Phi_{k;\beta}\bigl(0\leftrightarrow \Lambda_{k}^\comp\bigr) \Phi_{N-k';\beta}\bigl( 0\leftrightarrow \Lambda_{N-k'}^\comp \bigr)\\
			&\leq C'N^{3d-1} \sum_{k=N/3}^{2N/3} \Phi_{k;\beta}\bigl(0\leftrightarrow \Lambda_{k}^\comp\bigr) \Phi_{N-k;\beta}\bigl( 0\leftrightarrow \Lambda_{N-k}^\comp \bigr).
		\end{align*}
		Plugging this into our first bound on \(b_N\), we conclude that, for any \(N\) large enough,
		\[
		b_N \leq \sfe^{-c_1 N} + N^{3d} \sum_{k=N/3}^{2N/3} b_k b_{N-k},
		\]
		which is the first condition of Lemma~\ref{lem:stretch_exp_to_exp_seq}.
	\end{proof}
	Application of Lemma~\ref{lem:stretch_exp_to_exp_seq} concludes the proof.
\end{proof}

\begin{proof}[Proof of Lemma~\ref{lem:stretch_exp_seq}]
	Consider the sequence \(\tilde a_N = \max\{a_N, \sfe^{-c_1N/2}\}\).
	The second condition implies the existence of \(C_3 \geq 1\) and $N_{0}\geq 0$ such that, for any $N$ sufficiently large
	\[
		\tilde a_{m N} \leq C_1 N^{\alpha} \tilde a_N^2 + C_2\sfe^{-c_1 N} \leq (C_3N)^{\alpha} \tilde a_N^2.
	\]
	Define \(b_N= -\log \tilde a_N\). It is now sufficient to prove that \(b_N \geq cN^{\nu}\) for some \(c>0\) and all \(N\) large enough.
	The inequality above becomes
	\[
		b_{mN} \geq -\alpha\log(C_3 N) + 2 b_N.
	\]
	In particular, for \(k\geq 1$ and $N$ sufficiently large
	\begin{align*}
		b_{m^k N} &\geq 2^kb_{N} - \frac{\alpha}2 \sum_{i=1}^{k} 2^i\log(C_3Nm^{k-i})\\
		&= 2^k \bigl(b_{N} - \frac{\alpha}2\log(C_3N) \sum_{i=1}^{k} 2^{i-k} - \frac{\alpha}2\log m \sum_{i=1}^{k-1} 2^{i-k}(k-i)\bigr)\\
		&\geq 2^k\bigl(b_{N} - \frac{\alpha}2\log(C_3N) \sum_{i=0}^{\infty} 2^{-i} - \frac{\alpha}{4}\log m \sum_{i=0}^{\infty} 2^{-i}(i+1)\bigr)\\
		&= 2^k\bigl(b_{N} - \alpha\log(C_3N) -\alpha\log m \bigr).
	\end{align*}
where the first inequality follows from an easy induction. By our third assumption, \(b_{N_n}\geq (\log N_n)^{1+\epsilon}\) for any \(n\geq 1\). Let then \(N_0\) be such that \(b_{N_0} \geq \alpha \log(C_3 N_0 m)+ 1\), so that
	\[
		\forall k\geq 1,\quad b_{m^kN_0} \geq  2^k.
	\]
	By our first assumption, \(b_N\) is non-decreasing. Let \(\nu = \frac{\log 2}{\log m}\). Set \(c = (mN_0)^{-\nu}\). For any \(N\geq N_0\), one can find \(k\geq 1\) such that \(m^{k-1} N_0 \leq N < m^k N_0\). Hence,
	\[
		b_{N} \geq b_{m^{k-1}N_0} \geq 2^{k-1} = \frac1{(mN_0)^{\nu}}(m^{k}N_0)^{\nu} \geq c N^{\nu}.
		\qedhere
	\]
\end{proof}

\begin{proof}[Proof of Lemma~\ref{lem:stretch_exp_to_exp_seq}]
	From the first condition, we obtain
	\[
		a_{N}\leq \sfe^{-c_1 N} + N^{\alpha+1} \max_{k\in\{\lceil N/3\rceil, \dots, \lceil 2N/3\rceil\}} a_{k} a_{N-k},
	\]
	for all \(N\) large enough.
	By our second condition, there exists \(N_0 \geq 2\) such that
	\begin{itemize}
		\item the previous inequality holds for \(N\geq N_0\),
		\item \(\max_{N_0\leq k \leq 3N_0} (16k)^{\alpha + 1}a_{k} \leq \frac12\sfe^{-3}\),
		\item \(2(16N)^{\alpha +1} \sfe^{-c_1 N} \leq \frac12 \sfe^{-N/N_0}\) for \(N\geq N_0\).
	\end{itemize}
	We now claim that \(2(16N)^{\alpha + 1}a_{N}\leq \sfe^{-N/N_0}\) for any \(N\geq N_0\). We proceed by induction over \(N\). The cases \(N\in\{N_0,\dots,3N_0\}\) follow by the second bullet point above.
	Suppose now that the claim holds up to \(N\geq 3N_0\). Let us prove that it also holds for \(N+1\).
	\begin{align*}
		2(16(N+1))^{\alpha +1} a_{N+1}
		&\leq \frac12 \sfe^{-(N+1)/N_0} +  2(16(N+1))^{\alpha +1} (N+1)^{\alpha+1}\max_{k} a_k a_{N+1-k}\\
		&\leq \frac12 \sfe^{-(N+1)/N_0} +  \frac12 \max_{k} 2(16k)^{\alpha +1} a_k 2(16(N+1-k))^{\alpha+1} a_{N+1-k}\\
		&\leq \frac12 \sfe^{-(N+1)/N_0} +  \frac12 \max_{k} \sfe^{-k/N_0} \sfe^{-(N+1-k)/N_0} = \sfe^{-(N+1)/N_0},
	\end{align*}
	where we used the induction hypothesis in the last line and the maxima are over \(\lceil (N+1)/3 \rceil\leq k\leq \lceil 2(N+1)/3 \rceil\); in particular, for these choices of \(k\), \(\frac{N+1}{k}\leq 4\) and \(\frac{N+1}{N+1-k}\leq 4\).
\end{proof}

\subsection{Ratio mixing: Proof of Corollary~\ref{cor:main_mixing}}
Let us write \(\Phi=\Phi_\beta\). We first prove the claim for finite \(F,F'\). It is sufficient to show that, for any \(\eta\in \{0,1\}^{F}\) and \(\eta'\in \{0,1\}^{F'}\),
\[
	(1-\epsilon)^{-1} \geq \frac{\Phi(\omega_F=\eta)}{\Phi(\omega_F=\eta \given \omega_{F'} = \eta')} \geq (1+\epsilon)^{-1}
\]
with
\[
	\epsilon \equiv \epsilon(F,F',C,c) = \sum_{x\in V_F,y\in V_{F'}} C\sfe^{-c\norm{x-y}}.
\]
Let us first prove the following result.
\begin{lemma}\label{lem:mixing}
	Let \(\beta<\betac\). There exist \(C<\infty\) and \(c>0\) such that, for any \(F,F'\) finite and \(\eta\in\{0,1\}^{F}\),
	\[
		1-\epsilon(F,F',C,c) \leq \frac{\Phi\bigl(\omega|_F=\eta\bgiven \omega|_{F'} = 1\bigr)}{\Phi\bigl(\omega|_F=\eta\bgiven \omega|_{F'} = 0\bigr)}\leq 1+2\epsilon(F,F',C,c),
	\]
	whenever \(\epsilon\leq 1/2\). The same holds if one replace exactly one of \(\{\omega|_{F'} = 0\}\) or \(\{\omega|_{F'} = 1\}\) by \(\{\omega|_{F'} = \eta'\}\) for any \(\eta'\in \{0,1\}^{F'}\).
\end{lemma}
\begin{proof}
	Let \(\Xi\) be a monotone coupling of \(\Phi(\cdot \given \omega|_{F'} = 1)\) and \(\Phi(\cdot \given \omega|_{F'} = 0)\) such that, if \((\omega^+,\omega^-)\sim \Xi\), \(\omega^+\geq \omega^-\), one has that \(\omega^+\) and \(\omega^-\) agree on the complement of the cluster of \(V_{F'}\) in \(\omega^+\) (see the Appendix A in~\cite{Ott+Velenik-2018} for the proof of existence of such a coupling).
	Then,
	\begin{align*}
		\Xi\bigl(\omega|_F=\eta\bgiven \omega|_{F'} = 1\bigr) &= \Xi(\omega^+_F=\eta )\\
		&= \Xi(\omega^+_F=\eta, V_F\leftrightarrow_{\omega^+} V_{F'} ) + \Xi(\omega^+_F=\eta, V_F\nleftrightarrow_{\omega^+} V_{F'} )\\
		&= \Xi(\omega^+_F=\eta, V_F\leftrightarrow_{\omega^+} V_{F'} ) + \Xi(\omega^-_F=\eta, V_F\nleftrightarrow_{\omega^+} V_{F'} )\\
		& \leq \Phi\bigl(V_F\leftrightarrow V_{F'}, \omega|_F=\eta \bgiven \omega|_{F'} = 1\bigr) + \Phi\bigl(\omega|_F=\eta\bgiven \omega|_{F'} = 0\bigr).
	\end{align*}
	The uniform exponential decay of connectivities now implies that
	\[
		\Phi\bigl(V_F\leftrightarrow V_{F'},  \omega|_F=\eta \bgiven \omega|_{F'} = 1\bigr) \leq \Phi\bigl(\omega|_F=\eta \bgiven \omega|_{F'} = 1\bigr) \sum_{x\in V_F, y\in V_{F'}} C\sfe^{-c\norm{x-y}},
	\]
	which yields the upper bound. The same procedure applies if one does the replacements mentioned in the statement.
	
	To obtain the lower bound, let us write \(F_{*}=\setof{e\in F}{\eta_e=*}, *\in\{0, 1\}\). Then, the ratio we want to lower bound can be expressed as
	\begin{multline*}
		\frac{\Phi\bigl(\omega|_{F_1}=1 \bgiven \omega|_{F'} = 1\bigr) \Phi\bigl(\omega|_{F_0}=0 \given \omega|_{F_1}=1, \omega|_{F'} = 1\bigr)}{\Phi\bigl(\omega|_{F_1}=1 \bgiven \omega|_{F'} = 0\bigr) \Phi\bigl(\omega|_{F_0}=0 \bgiven \omega|_{F_1}=1, \omega|_{F'} = 0\bigr)}\\
		\geq \frac{ \Phi\bigl(\omega|_{F_0}=0 \bgiven \omega|_{F_1}=1, \omega|_{F'} = 1\bigr)}{\Phi\bigl(\omega|_{F_0}=0 \bgiven \omega|_{F_1}=1, \omega|_{F'} = 0\bigr)}.
	\end{multline*}
	Now,
	\begin{align*}
		\Phi\bigl(\omega|_{F_0}=0 \bgiven \omega|_{F_1}=1, \omega|_{F'} = 1\bigr) &\geq \Phi\bigl(\omega|_{F_0}=0, V_F\nleftrightarrow V_{F'} \bgiven \omega|_{F_1}=1, \omega|_{F'} = 1\bigr)\\
		&\geq \Phi\bigl(\omega|_{F_0}=0 \bgiven \omega|_{F_1}=1, \omega|_{F'} = 0\bigr)\bigl(1-\epsilon(F,F')\bigr),
	\end{align*}
	by monotonicity and the uniform exponential decay of connectivities. Again, the same procedure applies if one does the replacements mentioned in the statement.
\end{proof}

To get Corollary~\ref{cor:main_mixing} from there, let us write
\[
	\frac{\Phi(\omega_F=\eta)}{\Phi(\omega_F=\eta \given \omega_{F'} = \eta')} = \sum_{\tau\in\{0,1\}^{F'}} \Phi(\omega_{F'}=\tau) \frac{\Phi(\eta \given \tau)}{\Phi(\eta \given 0)}\frac{\Phi(\eta \given 0)}{\Phi(\eta \given 1)}\frac{\Phi(\eta \given 1)}{\Phi(\eta \given \eta')},
\]
where we have written \(\Phi(\eta\given *) = \Phi(\omega_{F}=\eta \given \omega_{F'}=*)\),
and use Lemma~\ref{lem:mixing}.
To get the case of finitely supported events \(A\) and \(B\), we can sum over configurations in \(A\) and \(B\) and apply the bound configuration-wise. To treat events in \(A\in\calF_{F}\) and \(B\in \calF_{F'}\) with \(F\) and \(F'\) infinite, we approximate the events \(A\) and \(B\) by events \(A_n\) and  \(B_n\) that are supported on finite sets \(F_n,F'_n\).
\(\epsilon(F,F',C,c)\) provides a uniform bound on \(\epsilon(F_n,F'_n,C,c)\). So,
\[
	 \frac{\Phi(A\cap B)}{\Phi(A)\Phi(B)} = \lim_{n\to\infty} \frac{\Phi(A_n\cap B_n)}{\Phi(A_n)\Phi( B_n)}
	 \begin{cases}
	 	\leq 1+ \epsilon(F,F',C,c),\\
	 	\geq 1- \epsilon(F,F',C,c).
	 \end{cases}
\]

\section{Asymptotics of connexion probabilities}

In all this section we work with \(\beta<\betac\). So, there is a unique infinite volume measure which is denoted \(\Phi_\beta\).

Recall that we defined \(\constPrefact_n(s)=\constPrefact_n(s,\beta,q)\) by
\begin{equation*}
	\constPrefact_n(s,\beta,q) = \frac{\beta}{q}\sfe^{\rho(ns)} \sum_{u,v\in \Zd} \Phi_{\beta}(0\leftrightarrow u) \sfe^{-\rho(ns-u-v)} \Phi_{\beta}(0\leftrightarrow v).
\end{equation*}
The goal of this section is the proof of Theorems~\ref{theorem:sharp_asymptotics} and~\ref{theorem:betasat=betahat}.

\subsection{Technical preparations}

We first state a few definitions/observations.

\subsubsection*{Pivotal edges}
Introduce the set \(\Piv_x(\omega)\) of edges pivotal in \(\omega\) for the event \(\{0\leftrightarrow x\}\).
\label{figure:Claim_3}
\subsubsection*{Nice connections}
Let \(f:\R^d \to \R_+\). Introduce the \emph{nice connection} event:
	\begin{multline}
		\label{eq:nice_connection_def}
			\niceConnection_x(f) = \{0\leftrightarrow x\}\cap \{|C_0|\leq f(x) \}\cap\\
		\cap \big\{ \{u,v\}\subset C_0 \ \mathrm{ AND }\ \rho(u-v)\geq 3\log f(x) \ \mathrm{ AND }\  \omega_{uv} = 1 \implies  \{u,v\}\in \Piv_x \big\}.
	\end{multline}
\begin{figure}[t]
	\centering
	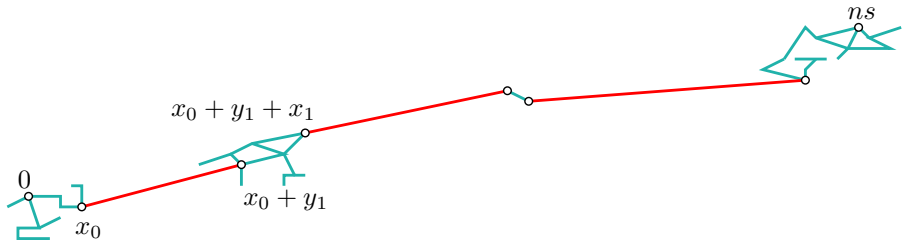
	\caption{Depiction of a realization of the nice connection event $\niceConnection_x$. Notice that all red edges are pivotal for the event $\lbrace 0\leftrightarrow x\rbrace$.}
	\label{figure:nice_connection}
\end{figure}	
	
These restricted connections allow a finer control of the geometry of typical clusters. The proof of the prefactor will be done by first reducing the analysis to this class of events and then proving the result for them.

\subsubsection*{Dual vector} Let \(s\in\bbS^{d-1}\). Let \(t\) be dual to \(s\). Suppose that \(\bbG_{\beta}(t)<\infty\).
We claim that
\begin{equation}
\label{eq:Gt_finite_implies_t_nu_dual}
	t\cdot x\leq \nu_{\beta}(x),
\end{equation}for any \(x\in \R^d\). Indeed, suppose \(t\cdot s'> \nu_{\beta}(s')\) for some \(s'\in \bbS^{d-1}\). Then, for any \(n\) large enough, \(t\cdot [ns'] \geq (1+\epsilon ) \nu_{\beta}([ns'])\) for some \(\epsilon>0\), where \([ns']\) is the lattice point closest to \(ns'\). Thus,
\begin{equation*}
	\bbG_{\beta}(t)\geq \sum_{n\geq 1} e^{-\nu_{\beta}([ns']) + \sfo(n)+ t\cdot [ns']} \geq \sum_{n\geq n_0} e^{\epsilon \nu_{\beta}([ns'])/2} = \infty.
\end{equation*}

We collect some intermediate results which will be at the core of the proof of Theorem~\ref{theorem:sharp_asymptotics}. We will again use the OSSS inequality, but closer to what is done in~\cite{Hutchcroft-2020}. We will use the following inequality.
\begin{lemma}\label{lem:OSSS_field_connect_volume}
	Let \(A\) be an increasing event, \(N\geq 1\), \(\lambda> 0\). Then,
	\[
		\sum_{e\in E_d} \Phi_\beta \bigl(\mathds{1}_{A}\mathds{1}_{|C_0|\geq N}\ ;\ \omega_e\bigr) \geq \frac12\Phi_\beta\bigl(A,|C_0|\geq N \bigr) \Bigl(\frac{1-\sfe^{-\lambda}}{\Phi_\beta\bigl(1-\sfe^{-\lambda|C_0|/N}\bigr)}-1\Bigr).
	\]
\end{lemma}
\begin{proof}
	The proof follows exactly the one in~\cite[Proposition 3.1]{Hutchcroft-2020}, with the following changes:
	\begin{itemize}
		\item use \(f=\mathds{1}_{A}\mathds{1}_{|C_0|\geq N}\) instead of \(f=\mathds{1}_{|C_0|\geq N}\),
		\item \(\Phi_\beta\bigl(1-\sfe^{-\lambda|C_0|/N}\bigr) = \sup_{v\in\Zd} \Phi_\beta\bigl(1-\sfe^{-\lambda|C_v|/N}\bigr)\) by translation invariance,
		\item use the inequality
		\begin{multline*}
			\Phi_\beta\Bigl(\bigl(1-\sfe^{-\lambda |C_0|/N}\bigr) \mathds{1}_{A}\mathds{1}_{|C_0|\geq N}\Bigr) -\Phi_\beta\bigl(1-\sfe^{-\lambda |C_0|/N}\bigr)\Phi_\beta\bigl(\mathds{1}_{A}\mathds{1}_{|C_0|\geq N}\bigr)\\
			\geq (1-\sfe^{-\lambda})\Phi_\beta\bigl(\mathds{1}_{A}\mathds{1}_{|C_0|\geq N}\bigr) -\Phi_\beta\bigl(1-\sfe^{-\lambda |C_0|/N}\bigr) \Phi_\beta\bigl(\mathds{1}_{A}\mathds{1}_{|C_0|\geq N}\bigr).
			\qedhere
		\end{multline*}
	\end{itemize}
\end{proof}
From this, we can deduce a bound on the volume of the connected component of \(0\) (which is the first step in comparing connections to nice connections). 
\begin{lemma}
	\label{lem:large_deviation_volume_connect}
	For any \(\beta<\beta'<\betac\), there exist \(c=c_{\beta,\beta'}>0\) and \(C= C_{\beta,\beta'}>0\) such that, for any \(x\in\Zd\),
	\begin{equation}\label{eq:large_deviation_volume_connect}
		\Phi_\beta\bigl(0\leftrightarrow x,|C_0|\geq N \bigr) \leq C \Phi_{\beta'}\bigl(0\leftrightarrow x,|C_0|\geq N \bigr) \sfe^{-cN},
	\end{equation}
	for any \(n\geq 0, N\geq 1\).
	In particular, we have
	\begin{equation}\label{eq:large_deviation_volume}
		\Phi_\beta\bigl(0\leftrightarrow x,|C_0|\geq N \bigr)\leq C\sfe^{-\nu_{\beta'}(x)-cN}.
	\end{equation}
\end{lemma}
\begin{proof}
First, observe that~\eqref{eq:large_deviation_volume} follows from~\eqref{eq:large_deviation_volume_connect} by noting that 
\(\Phi_{\beta'}\bigl(0\leftrightarrow x \bigr)\leq \sfe^{-\nu_{\beta'}(x)}\) (by sub-additivity). We will therefore focus on proving~\eqref{eq:large_deviation_volume_connect}. For \(\beta<\betac\), we have exponential decay of connectivities in finite volume, uniformly over boundary conditions. In particular, for any \(x\) and any \(N\geq 1\), \(\Phi_\beta\bigl(0\leftrightarrow x,|C_0|\geq N \bigr)\) is differentiable in \(\beta\) on the interval \((0,\betac)\). Moreover, it follows from Lemma~\ref{lem:OSSS_field_connect_volume} that, for any \(N\geq 1,\lambda>0\),
	\begin{align*}
		\frac{\rmd}{\rmd\beta} \Phi_\beta\bigl(0\leftrightarrow x,|C_0|\geq N \bigr)
		&= \sum_{e\in E_d} \frac{J_e}{1-\sfe^{-\beta J_e}} \Phi_\beta\bigl(0\leftrightarrow x,|C_0|\geq N\ ;\ \omega_e \bigr)\\
		&\geq \frac1{2\beta} \Phi_\beta\bigl(0\leftrightarrow x,|C_0|\geq N \bigr) \Bigl(\frac{N(1-\sfe^{-\lambda})}{\lambda \Phi_\beta\bigl(|C_0|\bigr)}-1\Bigr),
	\end{align*}
	where we used
	\[
		\Phi_\beta\bigl(1-\sfe^{-\lambda|C_0|/N}\bigr)
		\leq \frac{\lambda}{N} \Phi_\beta(|C_0|).
	\]
	By taking the limit \(\lambda\downarrow 0\), we deduce that, for any \(0<\beta<\betac\), \(x\in \Zd\), and \(N\geq 1\),
	\[
		\frac{\rmd}{\rmd\beta} \log\Phi_\beta\bigl(0\leftrightarrow x,|C_0|\geq N \bigr) \geq \frac1{2\beta}\Bigl(\frac{N}{\Phi_\beta\bigl(|C_0|\bigr)}-1\Bigr).
	\]
	Integrating this differential inequality yields: for \(\beta<\beta'<\betac\),
	\begin{equation}\label{eq:OSSSexpovolume}
		\Phi_\beta\bigl(0\leftrightarrow x,|C_0|\geq N \bigr) \leq \Phi_{\beta'}\bigl(0\leftrightarrow x,|C_0|\geq N \bigr) C_{\beta, \beta'} \sfe^{-c_{\beta, \beta'} N},
	\end{equation}
	where \(c_{\beta,\beta'} = \frac{\beta'-\beta}{2\beta'\Phi_{\beta'}(|C_0|)}\) and \(C_{\beta,\beta'} = \sfe^{(\beta'-\beta)/2\beta'}\) (note that \(\Phi_{\beta'}(|C_0|)<\infty\)).
\end{proof}
Next, we need to control the pivotality of ``long'' edges. This is the content of the next Lemma.

\begin{lemma}\label{lem:long_edges}
	Assume that \(J\) is exponentially-decaying. Let \(\beta<\betac\).
	Then, there exist \(C<\infty,c>0\) such that, for any \(x\in \Zd\) and \(N\geq 1\),
	\begin{multline}
		\label{eq:longEdgesPivotal_conditioned_clusterSize}
			\Phi_\beta\bigl(\exists \{u,v\}\notin\Piv_{x}, \{u,v\}\subset C_0, \omega_{uv}=1, \rho(v-u) \geq 3\log N \given 0\leftrightarrow x,|C_0|\leq N \bigr)\\\leq CN^{-1/2}.
	\end{multline}
\end{lemma}
\begin{proof}
	Let \(\beta<\betac\). Let \(N\geq 1\).
	We want to control the number of ``long'' edges. Let us introduce the random variable
	\[
	L_{K, x} = \#\bsetof{\{i,j\}\subset C_0}{\rho(i-j)\geq K, \{i,j\}\notin \Piv_x, \omega_{ij} = 1}.
	\]
	Then, for any \(K\) large enough,
	\begin{align*}
		&\Phi_\beta\bigl( L_{K,x}>0 \bgiven
		0\leftrightarrow x, |C_0|\leq N\bigr) \\
		&\quad \leq \sum_{u\in\Zd} \sum_{v:\,\rho(u-v)\geq K} \!\!\!\!\Phi_\beta\bigl( u\leftrightarrow 0, \{u,v\}\notin \Piv_x, \omega_{uv} =1 \bgiven 0\leftrightarrow x, |C_0|\leq N \bigr)\\
		&\quad \leq \sum_{u\in\Zd}\sum_{v:\,\rho(u-v)\geq K} (1-\sfe^{-\beta J_{uv}}) \Phi_\beta\bigl( u\leftrightarrow 0 \bgiven
		0\leftrightarrow x, |C_0|\leq N\bigr)\\
		&\quad \leq \Bigl(\sum_{v:\,\rho(v)\geq K} \beta J_{0v}\Bigr) \Phi_\beta\bigl( |C_0| \bgiven
		0\leftrightarrow x, |C_0|\leq N\bigr)
		\leq \beta C \sfe^{-K/2} N,
	\end{align*}
	where we used the definition of \(J\). Taking \(K=K_N=3\log N\), one obtains the claim,
\end{proof}

We then have a BK type property for connections using pivotal edges.
\begin{lemma}\label{lem:BK_piv}
	Let \(e_i=(x_i,y_i), i=1,\dots, k\), be oriented edges. Let \(z\in\Zd\). Then, for any \(\beta \geq 0\) and \(q\geq 1\),
	\begin{equation*}
		\Phi_{\beta}\big(0\leftrightarrow z, D(e_1,\cdots, e_k)\big) \leq \Phi_{\beta}(0\leftrightarrow x_1) \prod_{i=1}^{k} \frac{\sfe^{\beta J_{x_iy_i}}-1}{q}\Phi_{\beta}(y_i\leftrightarrow x_{i+1}),
	\end{equation*}where \(x_{k+1} = z\), and \(D(e_1,\cdots, e_k)\) is the event that any self avoiding path of open edges from \(0\) to \(z\) passes through \(e_1,\cdots, e_k\) in that order (as a sequence of oriented edges, in particular \(\{x_i,y_i\}\in \Piv_z\)).
\end{lemma}
\begin{proof}
	Denote \(x_{k+1}= z\). Let \(A\) be the event that no two elements of \(\{x_1,\dots, x_k, x_{k+1}\}\) are connected together. Then,
	\begin{equation*}
		\Phi_{\beta}\big(0\leftrightarrow z, D(e_1,\cdots, e_k) \big) \leq \Phi_{\beta}(0\leftrightarrow x_1, A, y_i\leftrightarrow x_{i+1}, i=1,\dots,k) \prod_{i=1}^k\frac{\sfe^{\beta J_{x_iy_i}}-1}{q}
	\end{equation*}where we used the definition of the Random-Cluster measure to close the \(e_i\)s and inclusion of events. Monotonicity then implies
	\begin{equation*}
		\Phi_{\beta}(0\leftrightarrow x_1, A, y_i\leftrightarrow x_{i+1}, i=1,\dots,k) \leq \Phi_{\beta}(0\leftrightarrow x_1) \prod_{i=1}^{k} \Phi_{\beta}(y_i\leftrightarrow x_{i+1}).
		\qedhere
	\end{equation*}
\end{proof}

The final preliminary result will give a splitting of the cluster contributing to \(\niceConnection_x\) (which is the reason why we are interested in them in the first place).
\begin{lemma}
\label{lem:nice_co_splitting}
	Let \(\tilde{f}:\R_+\to \R_+\). Define \(f:\Zd\to \R_+\) by \(f(x) = \tilde{f}(\rho(x))\). Suppose that for \(\rho(x)\) large enough one has
	\(\frac{\rho(x)}{f(x)}\geq 3\log f(x)\). Then, for any \(x\in \Zd\) with \(\norm{x}\) large enough
	\begin{equation}
	\label{eq:nice_co_splitting}
		\Phi_{\beta}(\niceConnection_x(f)) \leq \sum_{k= 1}^{f(x)} \sum_{x_0,\dots, x_k}^* \sum_{y_1,\dots,y_k}^* \prod_{i=0}^k \Phi_{\beta}(0\leftrightarrow x_i) \prod_{j=1}^k \frac{e^{\beta J_{y_j}}-1}{q}
	\end{equation}where the \(*\) sums are over \(x_0,\dots,x_k\), and \(y_1,\dots,y_k\) satisfying
	\begin{itemize}
		\item \(\sum_{i=0}^k x_i + \sum_{j=1}^k y_j = x\),
		\item \(\rho(y_j)\geq 3\log f(x)\) for \(j=1,\dots, k\),
		\item \(\sum_{i=0}^k\rho(x_i)\leq f(x)3\log f(x)\).
	\end{itemize}
\end{lemma}
\begin{proof}
	The event \(\{ 0\leftrightarrow x \}\) implies the existence of a self-avoiding path \(\gamma:0\to x\) formed of open edges. Note that any edge in \(\Piv_x\) belongs to that path. Fix some arbitrary way of choosing a path from a configuration. We split \(\gamma(\omega)\) as follows:
	\begin{equation*}
		\tau_0=-1,\quad \tau_i = \min\{ r> \tau_{i}:\ \rho(\gamma_{r+1}-\gamma_r)\geq 3\log f(x)\}.
	\end{equation*}Let \(k\) be the largest index for which \(\tau_k\) is defined and set \(\tau_{k+1}= |\gamma|\), \(\tau_{l} = \infty\) for \(l>k+1\). We then set \(x_i=\gamma_{\tau_{i+1}} - \gamma_{\tau_{i}+1} \) for \(i=0,\dots, k\), and \(y_j = \gamma_{\tau_{j}+1}- \gamma_{\tau_{j}}\) for \(j=1,\dots, k\). Under \(\niceConnection_x(f)\), \(\{\gamma_{\tau_i},\gamma_{\tau_{i}+1}\}\in\Piv_x\). Moreover, the condition \(\rho(x)\geq f(x)3\log f(x)\) implies that \(k\geq 1\). So, summing over the possibilities for \(k\), \(x_0,\dots,x_k\) and \(y_1,\dots,y_k\) and using Lemma~\ref{lem:BK_piv},
	\begin{equation*}
		\Phi_{\beta}(\niceConnection_x(f))\leq \sum_{k=1}^{f(x)} \sum_{\substack{x_0,\dots, x_k\\\sum_{i=0}^k \rho(x_i)\leq f(x)3\log f(x)}}\sum_{\substack{y_1,\dots, y_k\\ \rho(y_j)\geq 3\log f(x)}} \mathds{1}_{\bar{x} + \bar{y} = x}\prod_{i=0}^k \Phi_{\beta}(0\leftrightarrow x_i)\prod_{j=1}^k \frac{e^{\beta J_{y_j}}-1}{q}
	\end{equation*}
	where we used translation invariance and \(\bar{x} = \sum_{i=0}^k x_i\), \(\bar{y} = \sum_{j=1}^k y_j\). The constraint on the \(x_i\)s comes from the fact that \(|\gamma|\leq |C_0|\leq f(x)\), and each edge in \(\gamma\) which is not one of the \(y_j\)s is of \(\rho\)-length at most \(3\log f(x)\).
\end{proof}

\subsection{Prefactor: Lower bound}

The first Lemma (which we shall use again in the upper bound) is
\begin{lemma}
	Let \(s\in \bbS^{d-1}\). Let \(\beta <\betasat(s)\) and suppose that there exists \(t\) dual to \(s\) with \(\bbG_{\beta}(t)<\infty\). Then, the sequence \(\constPrefact_n(s,\beta,q)\) converges to \(\constPrefact(s)\), which is given by
	\begin{equation}
		\constPrefact(s,\beta,q) = \frac{\beta}{q} \sum_{u,v\in\Zd} \sfe^{t\cdot u}\Phi_{\beta}(0\leftrightarrow u) \sfe^{-g_{t,s}(u+v)} \sfe^{t\cdot v}\Phi_{\beta}(0\leftrightarrow v),
	\end{equation}where \(g_{t,s}(x)= \lim_{n\to \infty} \surcharge_t(ns-x)\), and \(t\) is any vector dual to \(s\) with \(\bbG_{\beta}(t)<\infty\).
\end{lemma}
\begin{proof}
	Let \(s,t,\beta\) be as in the statement. The first observation is that \(\surcharge_t(ns+y)\) is non-increasing in \(n\) for any \(y\in\Rd\). Indeed, for any \(m,k\geq 0\) and \(y\in \R^d\),
\begin{multline*}
	\surcharge_t((m+k)s+y) = \rho((m+k)s+y) - t\cdot ((m+k)s+y)\leq \\
	\leq \rho(ks+y) + \rho(ms) - t\cdot (ks+y) - mt\cdot s\\
	= \rho(ks+y) - t\cdot (ks+y) = \surcharge_t(ks+y),
\end{multline*}where we used the triangle inequality in the second line, and \(t\cdot s =\rho(s)\) (as \(t\) is dual to \(s\)) in the third line. In particular, the surcharge being non-negative, \(g_{t,s}(y)= \lim_{n\to \infty} \surcharge_t(ns-y)\) is well defined.
	
	We then prove that the sequence \((\constPrefact_n(s))\) converges, and identify the limit, \(\constPrefact(s)\). Observe that \(\constPrefact_n(s)\) can be rewritten as follows:
\[
	\frac{q}{\beta} \constPrefact_n(s) = \sum_{u,v\in\Zd} \sfe^{t\cdot u} \Phi_\beta(0\leftrightarrow u) \sfe^{-\surcharge_t(ns-u-v)} \sfe^{t\cdot v} \Phi_\beta(0\leftrightarrow v).
\]
Then, one has that the function \(f_n(u,v) = \sfe^{t\cdot u} \Phi_\beta(0\leftrightarrow u) \sfe^{-\surcharge_t(ns-u-v)} \sfe^{t\cdot v} \Phi_\beta(0\leftrightarrow v)\) is non-decreasing in \(n\). We can therefore use the Monotone Convergence Theorem to obtain
\begin{equation}
	\label{eq:lim_G_n_s}
	\lim_{n\to \infty} \frac{q}{\beta} \constPrefact_n(s) = \sum_{u,v\in \Zd} \lim_{n\to\infty} f_n(u,v) = \sum_{u,v\in \Zd} \sfe^{t\cdot u} \Phi_\beta(0\leftrightarrow u) \sfe^{-g_{t,s}(u+v)} \sfe^{t\cdot v} \Phi_\beta(0\leftrightarrow v),
\end{equation}which proves the claim.
\end{proof}
\begin{remark}
\begin{enumerate}
	\item When the norm \(\rho\) is a \(\calC^1\) function in a neighbourhood of $s$ , one has the simpler expression
	\[
		\constPrefact(s) = \frac{\beta}{q} \bbG_{\beta}(t)^2.
	\]
	Indeed, in that case,
	\(
		\surcharge_t(ns-y) = n \bigl( \rho(s-\tfrac1n y) - \rho(s) \bigr) - t\cdot y = \sfo(1)
	\),
	since \(\rho(s-\frac1n y) - \rho(s) = \grad\rho(s) \cdot \tfrac1n y + \sfo(\frac1n) = \frac1n t\cdot y + \sfo(\frac1n)\). This implies that \(g_{t,s}(y)=0\) for all \(y\in\Zd\).
	\item If the norm $\rho$ is not locally \(\calC^1\), then $g_{t,s}(x)\neq 0$ in general: for instance, if $\rho:=\abs{\cdot}_{1}$ and $t=s=e_{1}$, then for $x\in\mathbb{Z}^{d}$ with $x_{1}=0$,
	\begin{equation*}
	\surcharge_t(ns-x)=\abs{ne_{1}-x}_{1}-t\cdot (ns-x)=\abs{x}_{1}.
	\end{equation*}
	\item The previous remarks show that in general, the expression for $\constPrefact(s)$ depends on the local geometry of $\mathscr{U}$.
\end{enumerate}
\end{remark}
We can then turn to the lower bound.
\begin{lemma}\label{lem:prefactor_sharp_LB}
	Let \(s\in \bbS^{d-1}\). Suppose that \(\psi(x) \propto \rho(x)^{-\alpha}\) with \(\alpha >\alphasat(s)\), or \(\psi(x) \propto \sfe^{-\tilde{c}\rho(x)^{\eta} }\) with \(\tilde{c}>0, \eta\in (0,1)\). Let \(\beta <\betasat(s)\) and suppose that there exists \(t\) dual to \(s\) with \(\bbG_{\beta}(t)<\infty\). Then,
	\begin{equation*}
		\Phi_{\beta}(0\leftrightarrow ns) \geq \constPrefact(s) J_{ns} (1+\sfo_n(1)).
	\end{equation*}
\end{lemma}
\begin{proof}
	Take \(R_n\) satisfying
	\begin{itemize}
		\item \(R_n\) is monotone increasing and \(\lim_{n\to\infty }R_n= \infty\).
		\item \(\psi(ns+x) =\psi(ns)(1+\sfo_n(1)) \) for any \(x\in \Lambda_{R_n}\).
		\item \(R_n^{2d} J_{ns} = \sfo_n(1)\).
	\end{itemize}
	Define \(\Delta_1 = \Lambda_{R_n}(0)\) and \(\Delta_2=\Lambda_{R_n}(ns)\). We proceed in this way to make the proof easy to adapt to other coupling constants. We could use the explicit choice \(R_n=n^{\delta}\) with \(\delta\in (0,1)\) chosen in the following way: if \(\psi\) has a polynomial form, then fix any \(\delta\in (0,1)\). If \(\psi\) has a stretched exponential form, fix \(\delta = (1-\eta)/2\).

\begin{figure}[t]
	\centering
	\includegraphics{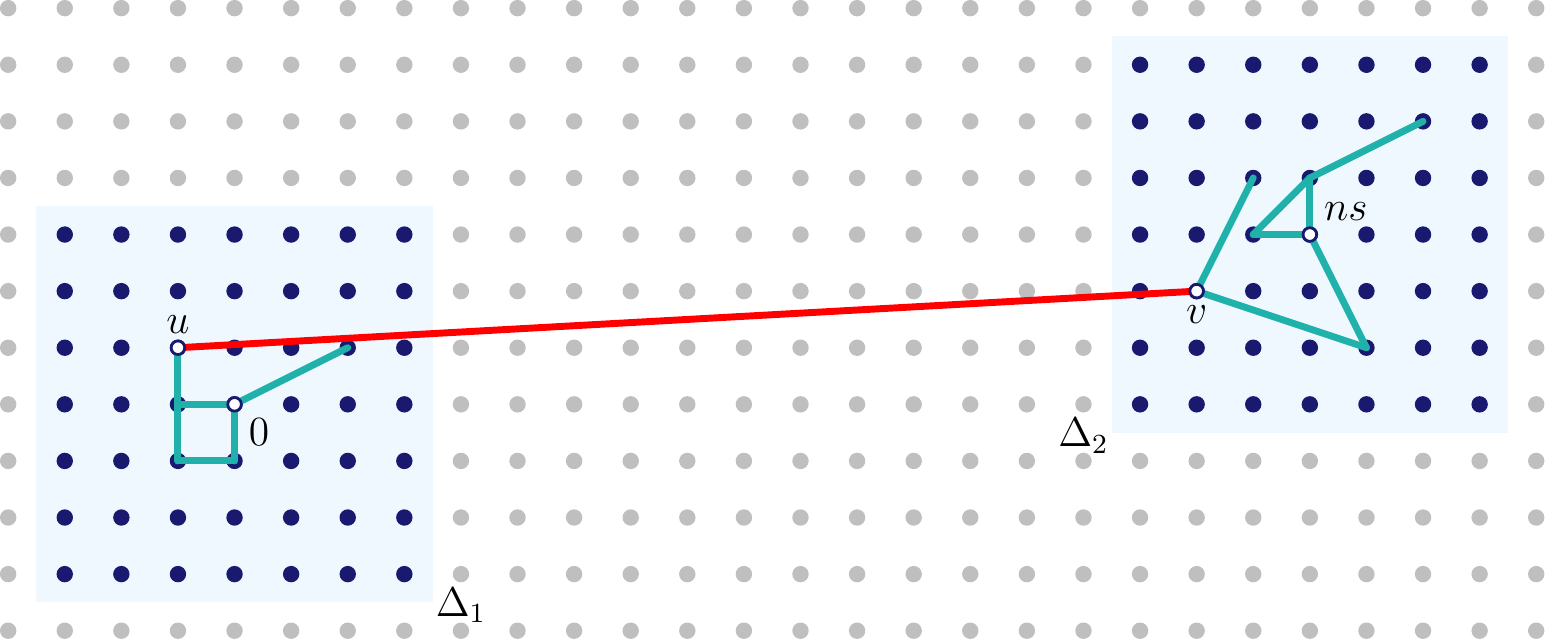}
	\caption{Setting in the proof of Lemma~\ref{lem:prefactor_sharp_LB}.}
	\label{figure:lower_bound_pre-factor}
\end{figure}		
	
	Let \(N\) denote the number of open edges from \(\Delta_1\) to \(\Delta_2\). Monotonicity and inclusion of events imply the following easy bound:
\[
	\Phi_{\Delta,\beta}^0(0\leftrightarrow ns, N=1) \leq \Phi_\beta(0\leftrightarrow ns). 
\]
Now, on the event \(\{N=1\}\), there is a unique open edge \(\{u,v\}\) from \(\Delta_1\) to \(\Delta_2\). By closing it (see Fig.~\ref{figure:lower_bound_pre-factor}, we obtain
\[
	\frac1{q} \sum_{u\in\Delta_1, v\in\Delta_2}
	\Phi_{\Delta,\beta}^0(0\leftrightarrow u, v\leftrightarrow ns, N=0) (\sfe^{\beta J_{u,v}}-1)
	=
	\Phi_{\Delta,\beta}^0(0\leftrightarrow ns, N=1).
\]
Now, conditionally on \(\{N=0\}\), the measure factorizes as follows:
\[
	\Phi_{\Delta,\beta}^0(0\leftrightarrow u, v\leftrightarrow ns \given N=0) = \Phi_{\Delta_1,\beta}^0(0\leftrightarrow u) \Phi_{\Delta_1,\beta}^0(0\leftrightarrow v).
\]
It is an easy consequence of the finite-energy property and of the choice of \(R_n\) that 
\[
	\Phi_{\Delta,\beta}^0(N\geq 1) \leq \sum_{u\in \Delta_1}\sum_{v\in \Delta_2}\beta J_{uv} \stackrel{n\to\infty}{\to} 0.
\]
Combining all these inequalities, we get
\begin{multline}\label{ineq:lower_bound_psi_hyp1}
	\sfe^{t\cdot ns} \Phi_\beta(0\leftrightarrow ns)
	\geq \\
	\frac{\beta}{q} \psi(ns) \sfe^{t\cdot ns} \sum_{u\in\Delta_1, v\in\Delta_1}
	\Phi_{\Delta_1, \beta}^0(0\leftrightarrow u) \sfe^{-\rho(ns-u-v)} \Phi_{\Delta_1, \beta}^0(0\leftrightarrow v) (1+\sfo_n(1)),
\end{multline}
where we used the definition of \(R_n\) as well as the fact \(\sfe^{\beta J_{u,v}}-1 = \beta J_{u,v} (1+\sfo_n(1))\) for any \(u\in \Delta_1\) and \(v\in \Delta_2\).
Now, by the monotone convergence theorem,
\begin{align*}
	\lim_{n\to\infty} \sfe^{t\cdot ns} \sum_{u,v\in\Delta_1}&
	\Phi_{\Delta_1, \beta}^0(0\leftrightarrow u) \sfe^{-\rho(ns-u-v)} \Phi_{\Delta_1, \beta}^0(0\leftrightarrow v)\\
	&= \lim_{n\to\infty} \sum_{u,v\in\Zd}
	\mathds{1}_{u,v\in\Delta_1}
	\sfe^{t\cdot u} \Phi_{\Delta_1, \beta}^0(0\leftrightarrow u)
	\sfe^{-\surcharge_t(ns-u-v)}
	\sfe^{t\cdot v} \Phi_{\Delta_1, \beta}^0(0\leftrightarrow v)\\
	&=\sum_{u,v\in\Zd}
	\sfe^{t\cdot u} \Phi_{\beta}(0\leftrightarrow u)
	\sfe^{-g_{s,t}(u+v)}
	\sfe^{t\cdot v} \Phi_{\beta}(0\leftrightarrow v) =  \constPrefact(s),
\end{align*}
which concludes the proof of the lower bound.
\end{proof}

\subsection{Prefactor: Upper bound}

The procedure here will be a combination of ``good enough bound implies the result'' (Lemma~\ref{lem:UB_good_imply_ok}) and ``bound enhancement'' (Lemma~\ref{lem:UB_enhancement}).

\begin{lemma}
	\label{lem:UB_good_imply_ok}
	Suppose \(\psi(x) \propto \rho(x)^{-\alpha}\) with \(\alpha>0\) or \(\psi(x)\propto \sfe^{-\tilde{c}\rho(x)^{\eta}}\) with \(\tilde{c}>0,\eta\in (0,1)\). Let \(s\in\bbS^{d-1}\). Let \(\beta'< \betasat(s)\). Suppose that there exists \(t\) dual to \(s\) with \(\bbG_{\beta'}(t)<\infty\). Suppose that, for any \(\epsilon>0\), there exist \(C,c\) such that, for all \(n\geq 1\),
	\begin{equation}
		\label{eq:assumed_decay_beta_prime}
		\sfe^{\rho(ns)}\Phi_{\beta'}(0\leftrightarrow ns) \leq C \psi(ns) e^{c \rho(ns)^{\epsilon}}.
	\end{equation}
	Then, for any \(\beta<\beta'\),
	\begin{equation*}
		\sfe^{\rho(ns)}\Phi_{\beta}(0\leftrightarrow ns) \leq \constPrefact(s,\beta,q) \psi(ns) (1+\sfo_n(1)).
	\end{equation*}
\end{lemma}
\begin{proof}
	Choose \(\epsilon\) in the following way: if \(\psi\) decays polynomially, take \(\epsilon <1\). If \(\psi(x)\propto \sfe^{-\tilde{c} \rho(x)^{\eta}}\), take \(\epsilon < 1-\eta\). Let \(\beta<\beta'\). Let \(f(x) = \kappa \rho(x)^{\epsilon}\). Then, using~\eqref{eq:large_deviation_volume_connect},
	one has that, for any \(\kappa>0\) and \(x\in \Zd\),
	\begin{equation*}
		\Phi_\beta\bigl(0\leftrightarrow x,|C_0|\geq f(x) \bigr) \leq C \Phi_{\beta'}\bigl(0\leftrightarrow x \bigr) \sfe^{-c\kappa \rho(x)^{\epsilon}},
	\end{equation*}
	so, using~\eqref{eq:assumed_decay_beta_prime}, \(\sfe^{t\cdot ns} \Phi_{\beta}\bigl(0\leftrightarrow ns,|C_0|\geq f(ns) \bigr)\leq \psi(ns)\sfo_n(1)\) for \(\kappa\) large enough. Fix such a value of \(\kappa\). By~\eqref{eq:longEdgesPivotal_conditioned_clusterSize},
	\begin{equation*}
		\Phi_{\beta}\bigl(0\leftrightarrow x,|C_0|\leq f(x) \bigr)\leq \Phi_{\beta}\bigl(\niceConnection_x(f) \bigr) + \sfo_{\rho(x)}(1)\Phi_{\beta}\bigl(0\leftrightarrow x,|C_0|\leq f(x) \bigr).
	\end{equation*}
	It is therefore sufficient to upper bound \(\Phi_{\beta}\bigl(\niceConnection_{ns}(f) \bigr)\). Using Lemma~\ref{lem:nice_co_splitting}, one obtains
	\begin{equation*}
		\Phi_{\beta}(\niceConnection_{ns}(f))\leq \sum_{k=1}^{f(ns)} \sum_{\substack{x_0,\dots, x_k\\\sum_{i=0}^k \rho(x_i)\leq f(ns)K_n}}\sum_{\substack{y_1,\dots, y_k\\ \rho(y_j)\geq K_n}} \mathds{1}_{\bar{x} + \bar{y} = ns}\prod_{i=0}^k \Phi_{\beta}(0\leftrightarrow x_i)\prod_{j=1}^k \frac{\sfe^{\beta J_{y_j}}-1 }{q},
	\end{equation*}
	where \(\bar{x} = \sum_{i=0}^k x_i \), \(\bar{y} = \sum_{j=1}^k y_j \), and \(K_n = 3\log(\kappa \rho(ns)^{\epsilon})\) diverges with \(n\). We then multiply both sides by \(\sfe^{\rho(ns)} = \sfe^{t\cdot ns} = \prod_{i=0}^k\sfe^{t\cdot x_i} \prod_{j=1}^k\sfe^{t\cdot y_j} \), and transpose the \(y_j\)s so that \(\rho(y_k)\geq \rho(y_j)\) to obtain that \(\sfe^{\rho(ns)}\Phi_{\beta}(\niceConnection_{ns}(f))(1+\sfo_n(1))\) is upper bounded by
	\begin{equation}
		\label{eq:UB_key_inequality}
		\sum_{k=1}^{f(ns)} 2^{k-1} k\sum_{\substack{x_0,\dots, x_k\\\sum_{i=0}^k \rho(x_i)\leq f(ns)K_n}}\sum_{\substack{y_1,\dots, y_k\\ \rho(y_j)\geq K_n\\ \rho(y_k)\geq \rho(y_j)}} \mathds{1}_{\bar{x} + \bar{y} = ns}\prod_{i=0}^k\sfe^{t\cdot x_i} \Phi_{\beta}(0\leftrightarrow x_i)\prod_{j=1}^k \sfe^{t\cdot y_j} \frac{\beta J_{y_j} }{q},
	\end{equation}
	where we used that for \(\rho(y)\geq K_n\), \(\sfe^{\beta J_{y}}-1 \leq (1+\sfo_n(1))\beta J_{y}\), and \((1+\sfo_n(1))^k \leq (1+\sfo_n(1)) 2^{k-1}\) for \(n\) large enough.
	
	We will now use the form of \(\psi\) through the following properties: there exist \(a,b \) such that for any \(u,v_1,\dots,v_k\) with \(\rho(u)\geq \rho(v_i)\),
	\begin{equation}	\label{eq:psi_property}
		\begin{gathered}
			\psi(u)\prod_{i=1}^k \psi(v_i)\leq b^{k} \psi(u+v)\prod_{i=1}^k \psi(v_i)^a,\\
			\bbJ_a(t) = \sum_{y\in \Zd} \sfe^{-\surcharge_t(y)} \psi(y)^a <\infty,\quad \psi(n+K_n\sfO(n^{\epsilon})) \leq \psi(n)(1+\sfo_n(1)),
		\end{gathered}
	\end{equation}
	where \(v= \sum_{i=1}^k v_i\). Let us first see that these hold.
	
	In the stretch exponential case (\(\psi(x) = c_{\tilde{c},\eta} \sfe^{-\tilde{c}\rho(x)^{\eta}}\)), one can take \(a= 2-2^{\eta} \), and \(b= (c_{\tilde{c},\eta})^{2^{\eta}-1} \). The first follows from triangular inequality and repeated use of
	\begin{equation*}
		c \sfe^{-\tilde{c}(n^{\eta}+m^{\eta} -(n+m)^{\eta})} = c \sfe^{-\tilde{c}m^{\eta} (\lambda^{\eta}+1 -(\lambda+1)^{\eta})} = c \sfe^{-\tilde{c}m^{\eta} (2 -2^{\eta})} = c^{2^{\eta}-1}\psi(m)^{2-2^{\eta}}
	\end{equation*}
	for \(n\geq m\), where \(\lambda = n/m \geq 1 \), \(c= c_{\tilde{c},\eta}\), as \((\lambda^{\eta}+1 -(\lambda+1)^{\eta})\) is increasing in \(\lambda\). The finiteness of \(\bbJ_a(t)\) follows from the super-polynomial decay of \(\psi(y)^a\). The last point follows from
	\begin{equation*}
		\sfe^{-\tilde{c} (n+\sfO(K_nn^{\epsilon}))^{\eta}} =
		\sfe^{-\tilde{c} n^{\eta}(1+\sfo(n^{-\eta}))^{\eta}} =
		\sfe^{-\tilde{c} n^{\eta}(1+\sfo(n^{-\eta}))} =
		\sfe^{-\tilde{c} n^{\eta}}e^{\sfo_n(1)} =
		\sfe^{-\tilde{c} n^{\eta}} (1+\sfo_n(1)),
	\end{equation*}
	by choice of \(\epsilon\) and logarithmic nature of \(K_n\).
	
	For the polynomial case (\(\psi(x) = c_{\alpha} \rho(x)^{-\alpha}\)), one can take \(a=1, b=2^{\alpha}\). The first property follows from repeated use of 
	\begin{equation*}
		\frac{(n+m)^{\alpha}}{n^{\alpha} m^{\alpha}}\leq 2^{\alpha}m^{-\alpha},
	\end{equation*}
	for \(n\geq m> 0\), and from triangular inequality. The bound \(\bbJ_a(t)=\bbJ(t)<\infty\) follows from finiteness of \(\bbG_{\beta}(t)\). Finally, \(c_{\alpha} (n+\sfO(n^{\epsilon}))^{-\alpha} = c_{\alpha} n^{-\alpha}(1+\sfo_n(1))^{-\alpha} = c_{\alpha} n^{-\alpha}(1+\sfo_n(1))\).
	We can now turn to the study of~\eqref{eq:UB_key_inequality}. Using~\eqref{eq:psi_property}, and that \(y_k\) is determined by \(\bar{x},y_1,\dots, y_{k-1}\) (due to \(\mathds{1}_{\bar{x} + \bar{y} = ns}\)), one obtains the upper bound
	\begin{multline}\label{eq:hypothesis_psi_property}
		\frac{\beta }{q} \sum_{k=1}^{f(ns)} k\sum_{\substack{x_0,\dots, x_k\\ \rho(\bar{x})\leq f(ns)K_n}}\sum_{\substack{y_1,\dots, y_{k-1}\\ \rho(y_j)\geq K_n}} \psi(ns-\bar{x})\sfe^{-\surcharge_t(ns-\bar{x}-\sum_{j=1}^{k-1}y_j)}\times\\
		\times \prod_{i=0}^k\sfe^{t\cdot x_i} \Phi_{\beta}(0\leftrightarrow x_i)  \prod_{j=1}^{k-1} \sfe^{-\surcharge_t(y_j)} \frac{2\beta b}{q} \psi(y_j)^a.
	\end{multline}
	We can then use \(\rho(\bar{x})\leq f(ns)K_n \leq C n^{\epsilon}\log(n)\) for some \(C\) uniform over \(\bar{x}\) and~\eqref{eq:psi_property} to obtain that, uniformly over \(\bar{x}\), \(\psi(ns-\bar{x})\leq \psi(ns)(1+\sfo_n(1))\). We can then treat separately the \(k=1\) term and the \(k\geq 2\) terms. Start with the latter. We have the upper bound (for \(n\) large enough)
	\begin{equation*}
		\frac{2\beta }{q} \psi(ns) \sum_{k=2}^{f(ns)} k \bbG_{\beta}(t)^{k+1}\Big(\sum_{y:\rho(y)\geq K_n} \sfe^{-\surcharge_t(y)} \frac{2\beta b}{q} \psi(y)^a\Big)^{k-1}.
	\end{equation*}
	As \(\bbJ_a(t)<\infty\), \(\lim_{n\to \infty} \sum_{y:\rho(y)\geq K_n} \sfe^{-\surcharge_t(y)} \psi(y)^a =0\), in particular, the whole sum over \(k\geq 2\) is \(\sfo_n(1)\) which implies that the \(k\geq 2\) terms are negligible. We now turn to the \(k=1\) term. It is upper bounded by
	\begin{multline*}
		\psi(ns)(1+\sfo_n(1))\frac{\beta }{q} \sfe^{t\cdot ns}\sum_{u,v} \Phi_{\beta}(0\leftrightarrow u)\sfe^{-\rho(ns-u-v)} \Phi_{\beta}(0\leftrightarrow v) =\\
		= \psi(ns)(1+\sfo_n(1)) \constPrefact_n(s,\beta,q),
	\end{multline*}
	where we used the definition of \(\surcharge_t\). Convergence of \(\constPrefact_n(s)\) implies the Lemma.
\end{proof}

\begin{lemma}
	\label{lem:UB_enhancement}
	Suppose \(\psi(x) \propto \sfe^{-\tilde{c} \rho(x)^{\eta}}\) with \(\tilde{c}>0,\eta\in (0,1)\). Let \(s\in\bbS^{d-1}\). Let \(\beta'< \betasat(s)\). Suppose that there exists \(t\) dual to \(s\) with \(\bbG_{\beta'}(t)<\infty\). Suppose that there exist \(C_1,c_1>0, \epsilon\in (0,1)\) such that
	\begin{equation}
		\sfe^{\rho(ns)}\Phi_{\beta'}(0\leftrightarrow ns) \leq C_1 \psi(ns) e^{c_1 \rho(ns)^{\epsilon}}.
	\end{equation}Then, for any \(\beta<\beta'\), there exist \(C_2,c_2>0\) such that
	\begin{equation*}
		\sfe^{\rho(ns)}\Phi_{\beta}(0\leftrightarrow ns) \leq C_2 \psi(ns) e^{c_2 \rho(ns)^{\epsilon-1+\eta}\log n }.
	\end{equation*}
\end{lemma}
\begin{proof}
	Let \(s,\beta,\beta',t\) be as in the statement of the lemma. Let \( f(x) = \kappa \rho(x)^{\epsilon}\). Using~\eqref{eq:large_deviation_volume_connect},
	one has that for any \(\kappa>0\) and \(x\in \Zd\),
	\begin{equation*}
		\Phi_\beta\bigl(0\leftrightarrow x,|C_0|\geq f(x) \bigr) \leq C \Phi_{\beta'}\bigl(0\leftrightarrow x \bigr) \sfe^{-c\kappa \rho(x)^{\epsilon}},
	\end{equation*}so, using our assumption, \(\sfe^{t\cdot ns} \Phi_{\beta}\bigl(0\leftrightarrow ns,|C_0|\geq f(ns) \bigr)\leq \psi(ns)\sfo_n(1)\) for \(\kappa\) large enough. Fix such a value of \(\kappa\). Proceeding as in the proof of Lemma~\ref{lem:UB_good_imply_ok}, one reduces to upper bound \(\Phi_{\beta}(\niceConnection_{ns})\), which is upper bounded by (via Lemma~\ref{lem:nice_co_splitting} and the same argument as in the proof of Lemma~\ref{lem:UB_good_imply_ok})
	\begin{equation*}
		\sum_{k=1}^{f(ns)} C^k k\sum_{\substack{x_0,\dots, x_k\\\sum_{i=0}^k \rho(x_i)\leq f(ns)K_n}}\sum_{\substack{y_1,\dots, y_k\\ \rho(y_j)\geq K_n\\ \rho(y_k)\geq \rho(y_j)}} \mathds{1}_{\bar{x} + \bar{y} = ns}\prod_{i=0}^k\sfe^{t\cdot x_i} \Phi_{\beta}(0\leftrightarrow x_i)\prod_{j=1}^k \sfe^{-\surcharge_t( y_j)}  \psi(y_j).
	\end{equation*}Still following the same procedure as in Lemma~\ref{lem:UB_good_imply_ok}, we use the property of \(\psi(x)\propto \sfe^{-\tilde{c}\rho(x)^{\eta} }\) that we used in Lemma~\ref{lem:UB_good_imply_ok} (see~\eqref{eq:psi_property}) with \(a=2-2^{\eta}\), and the estimate
	\begin{equation}\label{ineq:hypothesis_psi_2}
		\sfe^{-\tilde{c}\rho(ns-\bar{x})^{\eta} }\leq \sfe^{-\tilde{c}(\rho(ns)-Cn^{\epsilon}\log n)^{\eta} }= \sfe^{-\tilde{c}\rho(ns)^{\eta}(1-\sfO(n^{\epsilon-1}\log n)) }= \sfe^{-\tilde{c}\rho(ns)^{\eta} }\sfe^{\sfO(n^{\epsilon-1+\eta}\log n) },
	\end{equation}
	to obtain the upper bound
	\begin{multline*}
		C'\psi(ns)\sfe^{c n^{\epsilon-1+\eta}\log n } \sum_{k=1}^{f(ns)} C^k k\sum_{x_0,\dots, x_k}\sum_{\substack{y_1,\dots, y_{k-1}\\ \rho(y_j)\geq K_n}} \prod_{i=0}^k\sfe^{t\cdot x_i} \Phi_{\beta}(0\leftrightarrow x_i)\prod_{j=1}^{k-1} \sfe^{-\surcharge_t( y_j)}  \psi(y_j)^a =\\
		= C'\psi(ns)\sfe^{c n^{\epsilon-1+\eta}\log n } \sum_{k=1}^{f(ns)} C^k k \bbG_{\beta}(t)^{k+1} \Big( \sum_{y: \rho(y)\geq K_n} \sfe^{-\surcharge_t( y)}  \psi(y)^a \Big)^{k-1}.
	\end{multline*}
	As \(\lim_{n\to \infty } \sum_{y: \rho(y)\geq K_n} \sfe^{-\surcharge_t( y)}  \psi(y)^a = 0\), the sum over \(k\) is bounded uniformly over \(n\). This concludes the proof.
\end{proof}

We can now prove the upper bound of Theorem~\ref{theorem:sharp_asymptotics}.
\begin{proof}[Proof of Theorem~\ref{theorem:sharp_asymptotics}: upper bound]
	As \(\beta\) is supposed to be strictly smaller than \(\betahat(s)\), all hypotheses of Lemma~\ref{lem:UB_good_imply_ok} except~\eqref{eq:assumed_decay_beta_prime} are obviously fulfilled. When \(\psi(x)\propto \rho(x)^{-\alpha}\), \eqref{eq:assumed_decay_beta_prime} is also obviously satisfied. For \(\psi(x)\propto \sfe^{-\tilde{c}\rho(x)^{\eta}}\), repeated applications of Lemma~\ref{lem:UB_enhancement} yield the validity of~\eqref{eq:assumed_decay_beta_prime}.
\end{proof}

\subsection{Good dual vector}

We turn to the proof of Theorem~\ref{theorem:betasat=betahat}. The goal is to provide some control over the generating function \(\bbG_{\beta}(t)\) for suitable \(t\)s when \(\beta<\betasat(s)\), and \(\psi\) decays fast enough. Namely, we prove
\begin{theorem}
	Let \(s\in\bbS^{d-1}\). Suppose \(\psi(x)\leq C\rho(x)^{-2d-\epsilon}\) with \(\epsilon>0\). Then, for any \(\beta<\betasat(s)\), there exists \(t\) dual to \(s\) with \(\bbG_{\beta}(t)<\infty\).
\end{theorem}
\begin{proof}
	Let \(\beta<\beta'< \betasat(s)\).
	We first construct \(t\) with the desired properties. As \(\nu_{\beta'}\leq \rho\) (as norms), \(\mathscr{U}\subset \mathscr{U}_{\beta'}\) where \(\mathscr{U}_{\beta'}\) is the unit ball for \(\nu_{\beta'}\). Then, as \(\nu_{\beta'}(s)=\rho(s)\), \(\frac{s}{\rho(s)}\in (\partial\mathscr{U} \cap \partial\mathscr{U}_{\beta'})\). Let then \(H\) be a supporting hyperplane of \(\mathscr{U}_{\beta'}\) passing through \(\frac{s}{\rho(s)}\). It is also a supporting hyperplane of \(\mathscr{U}\). Let \(t'\) be orthogonal to \(H\) and such that \(t'\cdot s>0\). Set \(t = \frac{\rho(s)}{t'\cdot s}t'\). This is a dual vector of \(s\) by construction and it satisfies \(t\cdot x \leq \nu_{\beta'}(x)\) for any \(x\in \R^d\).
	
	By~\eqref{eq:large_deviation_volume} and the previous discussion,
	\begin{equation*}
		\Phi_{\beta}\big(0\leftrightarrow x, |C_0|\geq \kappa\log\rho(x)\big) \leq C\sfe^{-t\cdot x}\sfe^{-c\kappa \log\rho(x)}.
	\end{equation*}In particular, for \(\kappa\) large enough,
	\begin{equation*}
		\sum_{x\in \Zd} \sfe^{t\cdot x}\Phi_{\beta}\big(0\leftrightarrow x, |C_0|\geq \kappa\log\rho(x)\big) <\infty.
	\end{equation*}
	Denote \(f(x) = \kappa\log\rho(x)\). Then, by~\eqref{eq:longEdgesPivotal_conditioned_clusterSize}, for \(\norm{x}\) large enough,
	\begin{equation*}
		\Phi_{\beta}\big(0\leftrightarrow x, |C_0|\leq f(x)\big)\leq 2\Phi_{\beta}\big(\niceConnection_x(f)\big).
	\end{equation*}We are thus left with showing the summability of \(e^{t\cdot x}\Phi_{\beta}\big(\niceConnection_x(f)\big)\). Let \(L_x = (\log\rho(x))^{\delta}\), \(\delta = (d+1)/\epsilon\). Note that under \(\niceConnection_x(f)\), any open edge \(\{u,v\}\subset C_0\) with \(\rho(v-u)\geq L_x\) is in \(\Piv_{x}\), and \(C_0\) contains at least one such edge (for \(\rho(x)\) large enough). Proceeding exactly as in the proof of Lemma~\ref{lem:nice_co_splitting} (with \(L_x\) replacing \(3\log(\rho(x))\) in the definition of \(\tau_i\)), one obtains that for \(\rho(x)\) large enough
	\begin{equation*}
		\Phi_{\beta}(\niceConnection_x(f))\leq \sum_{k=1}^{f(x)} \sum_{\substack{x_0,\dots, x_k\\\sum_{i=0}^k \rho(x_i)\leq f(x)L_x}}\sum_{\substack{y_1,\dots, y_k\\ \rho(y_j)\geq L_x}} \mathds{1}_{\bar{x} + \bar{y} = x}\prod_{i=0}^k \Phi_{\beta}(0\leftrightarrow x_i)\prod_{j=1}^k \frac{2\beta J_{y_j}}{q},
	\end{equation*}where \(\bar{x}= \sum_{i=0}^k x_i\), \(\bar{y}= \sum_{j=1}^k y_j\), and we used \(e^{\beta J_{y}}-1 \leq 2\beta J_{y} \) for \(\rho(y)\) large enough. Multiplying both sides by \(\sfe^{t\cdot x}\), using \(\psi(x)\leq C\rho(x)^{-2d-\epsilon}\) and summing over \(x\) with \(\rho(x)\) large enough (symbolized by the \(*\) sum), one obtains
	\begin{multline*}
		\sum_{x}^* \sfe^{t\cdot x} \Phi_{\beta}(\niceConnection_x(f))\leq \\
		\leq C\sum_{x}^*  \sum_{k=1}^{f(x)} \sum_{\substack{x_0,\dots, x_k\\\sum_{i=0}^k \rho(x_i)\leq f(x)L_x}}\sum_{\substack{y_1,\dots, y_k\\ \rho(y_j)\geq L_x}} \mathds{1}_{\bar{x} + \bar{y} = x}\prod_{i=0}^k \sfe^{t\cdot x_i} \Phi_{\beta}(0\leftrightarrow x_i)\prod_{j=1}^k \sfe^{-\surcharge_t(y_j) } c \rho(y_j)^{-2d-\epsilon}.
	\end{multline*}
	Transposing the \(y_j\)s so that \(\rho(y_k)\geq \rho(y_j)\), and using the properties~\eqref{eq:psi_property} of \(\rho(x)^{-2d-\epsilon}\), the last display is upper bounded by
	\begin{equation*}
		C\sum_{x}^*  \rho(x)^{-2d-\epsilon} \sum_{k=1}^{f(x)} k c^k \sum_{\substack{x_0,\dots, x_k\\ \rho(x_i)\leq f(x)L_x}}\sum_{\substack{y_1,\dots, y_{k-1} \\ \rho(y_j)\geq L_x}} \prod_{i=0}^k \sfe^{t\cdot x_i} \Phi_{\beta}(0\leftrightarrow x_i)\prod_{j=1}^{k-1} \rho(y_j)^{-2d-\epsilon}.
	\end{equation*}Using then \(\sfe^{t\cdot x_i} \Phi_{\beta}(0\leftrightarrow x_i)\leq 1\), and the definitions of \(L_x,f(x)\), one has the upper bound
	\begin{equation*}
		C\sum_{x}^*  \rho(x)^{-2d-\epsilon}(\log \rho(x))^{2d(1+\delta)} \sum_{k\geq 1} k \Big( c (\log \rho(x))^{d(1+\delta)} (\log(\rho(x))^{\delta(d-2d-\epsilon)} \Big)^{k-1} .
	\end{equation*}where \(c,C\) now depend on \(\epsilon,\kappa,d,\beta,q\). Using finally \(\delta= (d+1)/\epsilon\), one has that the sum over \(k\) is summable uniformly over \(x\) with \(\rho(x)\) large. Moreover, \(\rho(x)^{-2d-\epsilon}(\log \rho(x))^{2d(1+\delta)}\) is also summable, which implies the claim.

\section{The case of subexponentially decaying coupling constants}
In this section, we explain how the proofs of the last section can be used to prove Theorem~\ref{thm:asymp_sub_exp_coupling}. We only highlight the differences.
Firstly, it was proved in~\cite{Hutchcroft-2020} that for $\beta<\betac$
\begin{equation*}
	\chi(\beta):=\sum_{x\in\mathbb{Z}^{d}}\Phi_{\beta}(0\leftrightarrow x)<\infty.
\end{equation*}
Morally, the modifications go as follows: set \(t=0\) and \(\nu_{\beta} = 0\) in the proofs. Note that in that case, $\bbG_{\beta}(t)= \chi(\beta)$, and \(\bbJ(t) =1\) (by our normalization choice).

To be more detailed, we present how to adapt the proofs claim by claim. We will need
\begin{equation*}
	F(r) = \sum_{x:\normsup{x}\geq r} J_x.
\end{equation*}Note that \(1\geq F(r)>0\) (as \(\sum_x J_x=1\)), \(F\) is strictly decreasing over \(\Z_+\) (as \(J_x>0\) for \(x\neq 0\)) and therefore invertible (as a function \(\Z_+\to \Z_+\)), and, by summability of \(J_x\), \(F(r)\to 0\) as \(r\to\infty\). We denote \(F^{-1}\) the extension to \(\R_+\) of the inverse of \(F\) by linear interpolation.

\begin{remark}
	As it was the case in the previous section, the proof of Theorem~\ref{thm:asymp_sub_exp_coupling} holds under more general assumptions on the coupling constants. We refer the interested reader to the Appendix.
\end{remark}

\subsection{Preparations} Lemma~\ref{lem:BK_piv} is valid in this context without any modification. In Lemma~\ref{lem:OSSS_field_connect_volume} as well as in the proof of Lemma~\ref{lem:large_deviation_volume_connect}, one can bypass differentiability problems by using Dini derivatives (as in~\cite{Hutchcroft-2020}). Lemma~\ref{lem:long_edges} has to be replaced by
\begin{lemma}\label{lem:sub_exp:long_edges}
	Let \(\beta<\betac\) and \(g:\R_+\to\R_+\) be decreasing to \(0\). Then, there exist \(C<\infty,c>0\) such that, for any \(x\in \Zd\) and \(N\geq 1\),
	\begin{multline}
		\label{eq:longEdgesPivotal_conditioned_clusterSize_Poly}
			\!\!\!\!\Phi_\beta\bigl(\exists \{u,v\}\notin\Piv_{x}, \{u,v\}\subset C_0, \omega_{uv}=1, \rho(v-u) \geq C F^{-1}(g(N)/N) \given 0\leftrightarrow x,|C_0|\leq N \bigr)\\\leq g(N).
	\end{multline}
\end{lemma}
The proof is a straightforward adaptation of the one of Lemma~\ref{lem:long_edges}. One also needs to change the notion of ``nice connections'':
Introduce the \emph{nice connection} event:
\begin{multline}
	\label{eq:nice_connection_def_subExp}
	\niceConnection_x(f) = \{0\leftrightarrow x\}\cap \{|C_0|\leq f(x) \}\cap\\
	\cap \big\{ \{u,v\}\subset C_0 \ \mathrm{ AND }\ \rho(u-v)\geq K_x \ \mathrm{ AND }\  \omega_{uv} = 1 \implies  \{u,v\}\in \Piv_x \big\},
\end{multline}where \(K_x= K_x(f) = CF^{-1}(1/f(x)\log f(x))\), with \(C\) given by Lemma~\ref{lem:sub_exp:long_edges} for \(g(N) = 1/\log N\).

\begin{remark}
	In the case of polynomially decaying interactions, \(F(r) \leq c r^{d-\alpha}\) and \(K_x \asymp (f(x)\log f(x))^{1/(\alpha-d)} \).
\end{remark}

Lemma~\ref{lem:nice_co_splitting} then becomes
\begin{lemma}
	\label{lem:nice_co_splitting_subExp}
	Let \(\tilde{f}:\R_+\to \R_+\). Define \(f:\Zd\to \R_+\) by \(f(x) = \tilde{f}(\rho(x))\). Suppose that for \(\rho(x)\) large enough one has
	\(\frac{\rho(x)}{f(x)}> K_x(f) =K_x\). Then, for any \(x\in \Zd\) with \(\norm{x}\) large enough
	\begin{equation}
		\label{eq:nice_co_splitting_subExp}
		\Phi_{\beta}(\niceConnection_x(f)) \leq \sum_{k= 1}^{f(x)} \sum_{x_0,\dots, x_k}^* \sum_{y_1,\dots,y_k}^* \prod_{i=0}^k \Phi_{\beta}(0\leftrightarrow x_i) \prod_{j=1}^k \frac{e^{\beta J_{y_j}}-1}{q}
	\end{equation}where the \(*\) sums are over \(x_0,\dots,x_k\), and \(y_1,\dots,y_k\) satisfying
	\begin{itemize}
		\item \(\sum_{i=0}^k x_i + \sum_{j=1}^k y_j = x\),
		\item \(\rho(y_j)\geq K_x\) for \(j=1,\dots, k\),
		\item \(\sum_{i=0}^k\rho(x_i)\leq f(x)K_x\).
	\end{itemize}
\end{lemma}

\subsection{Lower bound}
The lower bound follows closely the one of Lemma~\ref{lem:prefactor_sharp_LB} (and is the same as the proof of the lower bound in~\cite{Newman+Spohn-1998} for the Ising model). The only place where one has to be a bit careful is when choosing the size of the boxes \(\Delta_1,\Delta_2\).

\subsection{Upper bound}
We need to replace Lemmas~\ref{lem:UB_good_imply_ok} and~\ref{lem:UB_enhancement}. They are replaced by (the assumptions \(J_x\propto \rho(x)^{-\alpha}\) with \(\alpha>d\) or \(J_{x}\propto \sfe^{-\tilde{c}\rho(x)^{\eta}}\) are implicit)
\begin{lemma}
	\label{lem:UB_good_imply_ok_subPoly}
	Let \(\beta'< \betac\). Suppose that, for any \(\epsilon>0\), there exist \(C,c\) such that for any \(x\in \Zd\)
	\begin{equation}
		\label{eq:assumed_decay_beta_prime_subExp}
		\Phi_{\beta'}(0\leftrightarrow x) \leq C J_x e^{c \rho(x)^{\epsilon}}.
	\end{equation}
	Then, for any \(\beta<\beta'\),
	\begin{equation*}
		\Phi_{\beta}(0\leftrightarrow x) \leq \frac{\beta\chi(\beta)^{2}}{q} J_x (1+\sfo_{\norm{x}}(1)).
	\end{equation*}
\end{lemma}
The main difference in the proof is that, in the polynomial case, \(\epsilon>0\) has to be chosen small enough to be able to use Lemma~\ref{lem:nice_co_splitting_subExp}.

\begin{lemma}
	\label{lem:UB_enhancement_subExp}
	Suppose \(J_x \propto \sfe^{-\tilde{c} \rho(x)^{\eta}}\) with \(\tilde{c}>0,\eta\in (0,1)\). Let \(\beta'< \betac\). Suppose that there exist \(C_1,c_1>0, \epsilon\in (0,1)\) such that for any \(x\in \Zd\)
	\begin{equation}
		\Phi_{\beta'}(0\leftrightarrow x) \leq C_1 J_x e^{c_1 \rho(x)^{\epsilon}}.
	\end{equation}Then, for any \(\beta<\beta'\), there exist \(C_2,c_2>0\) such that
	\begin{equation*}
		\Phi_{\beta}(0\leftrightarrow x) \leq C_2 J_x e^{c_2 \rho(x)^{\epsilon-1+\eta}\log \rho(x) }.
	\end{equation*}
\end{lemma}

\appendix
\section{Existence of a saturation transition}\label{sec:Criterion}

In this appendix, we prove Theorem~\ref{thm:main_saturation_citerium}. As mentioned after its claim, the ``if'' part was proven in~\cite{Aoun+Ioffe+Ott+Velenik-CMP-2021}. For reasons explained in~\cite{Aoun+Ioffe+Ott+Velenik-CMP-2021}, establishing the claim for the self-avoiding walk automatically implies its validity for a general class of models, including the Random-Cluster model.
We start by briefly recalling the relevant definitions.

\medskip
A \emph{Self-Avoiding Walk} (SAW) is a sequence \(\gamma =(\gamma_0,\dots,\gamma_n)\in (\Zd)^{n+1}\) such that \(i\neq j\implies \gamma_i\neq\gamma_j\).
\(n\) is the \emph{length} of \(\gamma\) and is denoted by \(|\gamma|=n\).
For \(\lambda\in\R_+\), we associate to walks the \emph{weight}
\[
w_{\lambda}(\gamma) = \prod_{i=1}^{|\gamma|} \lambda J_{\gamma_i-\gamma_{i-1}},
\]
where the constants \((J_x)_{x\in\Zd}\) are those introduced in Section~\ref{ssec:Interactions} and are assumed to be exponentially-decaying.

\medskip
The \emph{two-point function} of the SAW is given by the partition function
\[
G_{\lambda}(x,y) = \sum_{\gamma:x\to y} w_{\lambda}(\gamma),
\]
where the notation \(\gamma:x\to y\) means that \(\gamma_0 = x\) and \(\gamma_{|\gamma|} = y\).

\begin{remark}
	As explained in Appendix~A.4.2 of~\cite{Aoun+Ioffe+Ott+Velenik-CMP-2021}, the link with the Random-Cluster model is obtained via the following lower bound on the two-point function of the latter model in terms of the two-point function of the SAW:
	\[
		\Phi_\beta^0(x\leftrightarrow y) \geq c_\beta G_{\lambda(\beta)}(x,y),
	\]
	with \(\lambda(\beta)\) satisfying \(\lim_{\beta\to 0} \lambda(\beta)=0\) (see~\cite{Aoun+Ioffe+Ott+Velenik-CMP-2021} for an explicit expression).
\end{remark}

The \emph{inverse correlation length} is defined as follows: for \(s\in\bbS^{d-1}\),
\[
\nu_\lambda(s) = -\lim_{n\to\infty} \frac1{n}\log G_\lambda(0,ns).
\]
The limit exists and its extension by positive homogeneity of order one defines a norm on \(\Rd\) (see~\cite{Aoun+Ioffe+Ott+Velenik-CMP-2021} for references and details). As for \(\rho\), introduce the associated convex set (Wulff shape)
\begin{equation*}
	\Wulff_{\lambda} = \setof{t\in\bbR^d}{\forall x\in\bbR^d,\, t\cdot x \leq \nu_{\lambda}(x)}
\end{equation*}which satisfies
\begin{equation}
	\label{eq:Wulff_to_norm}
	\nu_{\lambda}(x) = \max_{t\in \Wulff_{\lambda}} t\cdot x.
\end{equation}
As for \(\rho\), \(t\in \partial \Wulff_{\lambda}\) is said to be \(\nu_{\lambda}\)-dual to \(s\in\bbS^{d-1}\) if \(s\cdot t = \nu_{\lambda}(s)\).
The same argument as the one following~\eqref{eq:Gt_finite_implies_t_nu_dual} shows that \(\Wulff_{\lambda}\) is the closure of the convergence domain of the generating function
\begin{equation*}
	\bbG_{\lambda}(h) = \sum_{x\in\Zd} e^{h\cdot x} G_{\lambda}(0,x).
\end{equation*}
As before, we define the \emph{saturation point} by
\[
\lambdasat(s) = \inf\setof{\lambda\geq 0}{\nu_{\lambda}(s) < \rho(s)}.
\]

\medskip
Using the reduction explained in~\cite{Aoun+Ioffe+Ott+Velenik-CMP-2021}, Theorem~\ref{thm:main_saturation_citerium} is a consequence of the following result.
\begin{lemma}
	Suppose \(J\) is exponentially-decaying. Let \(s\in\bbS^{d-1}\). Suppose that, for all \(t\) \(\rho\)-dual to \(s\), \(\bbJ(t) = \infty\). Then \(\lambdasat(s) = 0\).
\end{lemma}
\begin{proof}
	Fix \(J\) and \(s\) as in the statement of the lemma. To prove the claim, it is sufficient to show that
	\[
	T_s = \setof{t\in\partial\Wulff}{t \text{ is } \rho\text{-dual to }s}
	\]
	is at positive distance from \(\partial\Wulff_{\lambda}\) for any \(\lambda>0\).
	Indeed, suppose this holds. Then there exists \(\epsilon>0\) and \(h\) \(\nu_{\lambda}\)-dual to \(s\) such that \((1+\epsilon)h \in (\Wulff\setminus \partial\Wulff)\). Hence, using~\eqref{eq:Wulff_to_norm},
	\[
	\rho(s)= \sup_{t\in\Wulff} t\cdot s \geq (1+\epsilon) h \cdot s > \nu_{\lambda}(s).
	\]
	This in turn implies that \(\lambdasat(s) < \lambda\) and the conclusion thus follows, since \(\lambda\) is arbitrary.
	
	Now, as \(T_s\) is compact (since the condition for \(t\) to be \(\rho\)-dual defines a closed subset of the compact set \(\calW\)), it is sufficient to show that, for any \(t\in T_s\), there exists \(\epsilon>0\) such that \(\bbG_{\lambda}\bigl((1-\epsilon)t\bigr) = \infty\) since $\Wulff_{\lambda}$ is the closure of the convergence domain of $\bbG_{\lambda}$.
	Fix \(t\in T_s\). The first observation is that one can find \(\delta>0\) arbitrarily small such that
	\[
	\cone = \setof{x\in\Zd}{\rho(x) - t\cdot x \leq \delta\norm{x}}
	\quad\text{and}\quad
	\cone_R = \setof{x\in\cone}{\norm{x}\leq R}
	\]
	satisfy the following two properties:
	\begin{itemize}
		\item there exists \(u\in\bbS^{d-1}\) such that \(\cone \setminus \{0\}\subset \setof{x}{x\cdot u > 0}\),
		\item \(\lim_{R\to\infty} \bbJ_{R}(t) = \infty\), where \(\bbJ_{R}(t) = \sum_{x\in\cone_R} J_x \sfe^{t\cdot x}\).
	\end{itemize}
	The first property follows from the fact that \(\rho\) is a norm: when \(\delta\downarrow 0\), \(\cone\) tends to the cone generated by the affine part of \(\partial\mathscr{U}\) towards which \(s\) points (which can be reduced to a point).
	The second is a direct consequence of our assumption that \(\bbJ(t)=\infty\) and the bound
	\[
	\sum_{x\in\Zd\setminus\cone} \psi(x) \sfe^{-(\rho(x) -t\cdot x)}
	\leq \sum_{x\in\Zd\setminus\cone} \psi(x) \sfe^{-\delta\norm{x}}
	< \infty.
	\]
	Now, observe that it follows from the strict inclusion of \(\cone\) in a half space that any concatenation of steps in \(\cone\) forms a self-avoiding walk. Moreover, by continuity, one also has
	\[
	\lim_{\epsilon\searrow 0} \bbJ_{R} \bigl((1-\epsilon)t\bigr) = \bbJ_{R}(t).
	\]
	Choosing \(R<\infty\) large enough and then \(\epsilon>0\) small enough, one obtains \(\bbJ_R((1-\epsilon) t)\geq \lambda^{-1}\).
	Therefore, for these choices of \(R\) and \(\epsilon\),
	\[
	\bbG_{\lambda}((1-\epsilon) t)
	\geq \sum_{n\geq 1} \lambda^n \sum_{y_1,\dots,y_n\in \cone_R} \prod_{i=1}^{n} J_{y_i} \sfe^{(1-\epsilon)t\cdot y_i}
	= \sum_{n\geq 1} \Bigl(\lambda\bbJ_{R}\bigl((1-\epsilon) t\bigr)\Bigr)^{\!n}
	= \infty,
	\]
	which concludes the proof.
\end{proof}

\section*{Acknowledgements}

YA and YV are supported by the Swiss NSF grant 200021\_200422. SO is supported by the Swiss NSF grant 200021\_182237. All authors are members of the NCCR SwissMAP.

\section*{Declarations}

The authors have no financial or proprietary interests in any material discussed in this article.

\bibliographystyle{plain}
\bibliography{BIGbib}

\end{document}

%% file: Images/UpperBound.pdf_t
\begin{picture}(0,0)%
\includegraphics{UpperBound.pdf}%
\end{picture}%
\setlength{\unitlength}{1160sp}%
\begingroup\makeatletter\ifx\SetFigFont\undefined%
\gdef\SetFigFont#1#2#3#4#5{%
  \reset@font\fontsize{#1}{#2pt}%
  \fontfamily{#3}\fontseries{#4}\fontshape{#5}%
  \selectfont}%
\fi\endgroup%
\begin{picture}(19008,5185)(1522,-6893)
\put(1801,-5731){\makebox(0,0)[lb]{\smash{{\SetFigFont{10}{12.0}{\familydefault}{\mddefault}{\updefault}{\color[rgb]{0,0,0}$0$}%
}}}}
\put(3016,-6721){\makebox(0,0)[lb]{\smash{{\SetFigFont{10}{12.0}{\familydefault}{\mddefault}{\updefault}{\color[rgb]{0,0,0}$x_0$}%
}}}}
\put(6571,-6091){\makebox(0,0)[lb]{\smash{{\SetFigFont{10}{12.0}{\familydefault}{\mddefault}{\updefault}{\color[rgb]{0,0,0}$x_0+y_1$}%
}}}}
\put(19351,-2131){\makebox(0,0)[lb]{\smash{{\SetFigFont{10}{12.0}{\familydefault}{\mddefault}{\updefault}{\color[rgb]{0,0,0}$ns$}%
}}}}
\put(5041,-4201){\makebox(0,0)[lb]{\smash{{\SetFigFont{10}{12.0}{\familydefault}{\mddefault}{\updefault}{\color[rgb]{0,0,0}$x_0+y_1+x_1$}%
}}}}
\end{picture}%